\documentclass{amsart}
\usepackage[utf8]{inputenc}
\usepackage[bookmarksnumbered, colorlinks, plainpages]{hyperref}

\newtheorem{theorem}{Theorem}[section]
\newtheorem{lemma}[theorem]{Lemma}
\newtheorem{proposition}[theorem]{Proposition}
\newtheorem{corollary}[theorem]{Corollary}
\theoremstyle{definition}
\newtheorem{definition}[theorem]{Definition}
\newtheorem{example}[theorem]{Example}

\newtheorem{remark}[theorem]{Remark}
\numberwithin{equation}{section}

\DeclareMathOperator*{\esssup}{ess\,sup}
\DeclareMathOperator*{\essinf}{ess\,inf}

\begin{document}

\title[Stability of trigonometric approximation in $L^p$]{Stability of trigonometric approximation in $L^p$ and applications to prediction theory}

\author[M. Frank, L. Klotz]{Michael Frank$^1$ and Lutz Klotz$^2$}

\address{$^1$Hochschule f\"ur Technik, Wirtschaft und Kultur (HTWK) Leipzig, Fakult\"at Informatik und Medien, PF 301166, D-04251 Leipzig, Germany.}
\email{michael.frank@htwk-leipzig.de}

\address{$^2$Universit\"at Leipzig, Fakult\"at f\"ur Mathematik und Informatik, Mathematisches Institut, Augustusplatz 10, D-04109 Leipzig, Germany.} 
\email{lutzklotz@t-online.de}

\subjclass[2010]{60625.}


\begin{abstract}
Let $\Gamma$ be an LCA group and $(\mu_n)$ be a sequence of bounded regular Borel measures on $\Gamma$ tending to a measure $\mu_0$. Let $G$ be the dual group of $\Gamma$, $S$ be a non-empty subset of $G \setminus \{ 0 \}$, and $[{\mathcal T}(S)]_{\mu_n,p}$ the subspace of $L^p(\mu_n)$, $p \in (0,\infty)$, spanned by the characters of $\Gamma$ which are generated by the elements of $S$. The limit behaviour of the sequence of metric projections of the function $1$ onto $[{\mathcal T}(S)]_{\mu_n,p}$ as well as of the sequence of the corresponding approximation errors are studied. The results are applied to obtain stability theorems for prediction of weakly stationary or harmonizable symmetric $p$-stable stochastic processes. Along with the general problem the particular cases of linear interpolation or extrapolation as well as of a finite or periodic observation set are studied in detail and compared to each other.
\end{abstract}

\maketitle

%=============================================%

\section{Introduction}
If a time series is modelled on a weakly stationary process, its prediction is calculated using the spectral measure of the process, which in turn is estimated on the basis of observations. Therefore, similar to the Central Limit Theorem giving a theoretical justification of statistical inference, it would be desirable to prove assertions claiming that minimally differing spectral measures yield only slightly varying prediction. According to \cite{[26]}, cf.~\cite{[28]}, Kolmogorov raised this problem at the VII-th Soviet Conference on Probability Theory and Mathematical Statistics in 1963. Let us formulate it more precisely. 

As usual denote the sets of positive integers, of non-negative integers, of integers, of real and of complex numbers by $\mathbb N$, ${\mathbb N}_0$, $\mathbb Z$, $\mathbb R$ and $\mathbb C$, respectively. Consider a weakly stationary process on $\mathbb R$ with a non-stochastic spectral measure $\mu_0$. Its linear prediction at the point zero on the basis of its observations at points of a set $S \subseteq {\mathbb R}$ is equivalent to the computation of the orthogonal projection $\phi_{\mu_0,2}$ of the function that is identically equal to $1$ onto the subspace of $L^2(\mu_0)$ spanned by the functions ${\rm e}^{{\rm i}x\cdot}$, $x \in S$. In most applications the measue $\mu_0$ is unknown and has to be replaced by its estimate $\mu$ obtained from data, cf.~\cite{[2]}. In fact, one computes the orthogonal projection $\phi_{\mu,2}$ of the function equal to $1$ in $L^2(\mu)$, and one hopes that $\phi_{\mu_0,2}$ and $\phi_{\mu,2}$ differ not too much. Accordingly, four kinds of prediction errors arise: 

(E1)  the theoretical or ''true'' prediction error $d_2(\mu_0)=\int_{\mathbb R} | 1-\phi_{\mu_0,2}|^2 \, d\mu_0$;

(E2)  the estimated error of the estimated prediction $d_2(\mu)=\int_{\mathbb R} | 1-\phi_{\mu,2}|^2 \, d\mu$;

(E3)  the theoretical error of the estimated prediction $d_2(\mu,\mu_0)=\int_{\mathbb R} | 1-\phi_{\mu,2}|^2 \, d\mu_0$;
 
(E4)  the estimated error of the theoretical prediction $d_2(\mu_0,\mu)=\int_{\mathbb R} | 1-\phi_{\mu_0,2}|^2 \, d\mu$.

Note, that additional assumptions are necessary to define (E3) and (E4) correctly, what will be explained in Section 3. The more, the theoretical prediction is not known and, thus, does not play any role in practice. So the error (E4) will be discussed only marginally. 

As the next step we will define what we mean by saying that $\mu_0$ and $\mu$ differ a little. Since both the measures are bounded it is natural to make use of the topology generated by the norm $\| \cdot \|$ on the linear space of $\mathbb C$-valued measures, where for a $\mathbb C$-valued measure $\nu$, the symbol $\|\nu\|$ denotes the total variation. Indeed, soon after A.~N.~Kolmogorov had raised the stability problem, Yu.~A.~Rozanov published the following partial answer in \cite{[28]}. We always denote by $\lim$ the limit of a sequence indexed by the elements of $\mathbb N$ in their natural order.

\begin{theorem} \label{t1.1}
   Let $(\mu_n)$ be a sequence of bounded Borel measures on $\mathbb R$. If $\mu_0 \leq \mu_n$ for all $n \in \mathbb N$ and $\lim \|\mu_0-\mu_n\|=0$, then $\lim d_2(\mu_n)=\lim d_2(\mu_n,\mu_0)= d_2(\mu_0)$. 
\end{theorem}

In course of his proof Yu.~A.~Rozanov obtained the following result:

\begin{theorem} \label{t1.2}
   Let $(\mu_n)$ be a sequence of bounded Borel measures on $\mathbb R$. If  $\,\lim \|\mu_0-\mu_n \|=0$, then $\lim \sup d_2(\mu_n) \leq d_2(\mu_0)$. 
\end{theorem}

Note, that among the four prediction errors (E1)-(E4), the error (E2) is the only error that can be computed. The preceding theorem as well as Theorem 3.1 and Example 3.2 below indicate that (E2) tends to be smaller (at least, not larger)  than the true prediction error. We are in the uncomfortable situation that doing prediction we must be afraid that the true prediction error is much larger than the prediction error we compute. 

The error (E3) gives the most objective idea on the correctness of the prediction. It was the behaviour of (E3) which was investigated in early papers on stability of prediction, cf.~\cite{[21],[28]}. Several results, e.g. Theorem 4.9, 5.10 and inequality (6.2) below reveal that, unlike (E2), the error (E3) tends to be larger than (E1), although this is not always the case, cf.~Example 3.6(a). This observation was the initiation to develop minimax robust methods, where the maximum of (E3) is computed under the condition that $\mu$ runs through a set of certain spectral measures, cf.~\cite{[11]} for a survey of early results. Such an approach to stability as well as the papers of M.~Moklyachuk, e.g. \cite{[24]} are not in the focus of the present paper. 

The preceding remarks make clear that  it is both of theoretical and practical importance to state sufficient conditions for the equalities  $\lim d_2(\mu_n)=d_2(\mu_0)$ or $\lim d_2(\mu_n,\mu_0)=d_2(\mu_0)$, in the style of Theorem \ref{t1.1}. From the point of view of applications it is probably even more important to prove convergence results for the sequence $(\phi_{\mu_n,2})$ of the predictions themselves. To the best of our knowledge, this question has not yet been discussed in the literature.

To point out the variety of possible assertions we decided to deal with further topologies on the space of $\mathbb C$-valued measures along with the norm topology. Section 2 is devoted to an overview of the corresponding modes of convergence we are interested in. However, for convenience of  presentation we confine ourselves to the study of sequences $(\mu_n)_{n \in \mathbb N}$ of measures although this is not quite the adequate approach in case of a non-metrizable topology. 

Stability of the prediction problem is equivalent to stability of the corresponding trigonometric approximation problem in $L^2(\mu_0)$. It is natural to ask for stability results of analogous approximation problems in $L^p(\mu_0)$, $p \in (0,\infty]$. Since Example 3.5 indicates that one can not hope for non-trivial affirmative assertions if $p=\infty$, this case will be excluded from our investigations. If $p\in (1,2]$, the results could be useful in prediction of harmonizable symmetric $p$-stable processes, cf.~\cite{[3],[31]}. For the other values of $p$ the results could be of interest from the point of view of trigonometric approximation. Recall in this context that Szeg\"o had studied his celebrated ''Szeg\"o infimum'' as a pure trigonometric approximation problem, cf.~\cite{[29]}. and only after the spectral theory of weakly stationary processes had been invented its applicability to prediction theory became clear, cf.\cite{[16],[17]}. 

Another natural extension is the study of a more general set of parameters than $\mathbb R$. Since there exists a well-developed Fourier theory on locally compact abelian groups, we shall be concerned with such a group as a parameter set of the process. In particular, our results are applicable to stochastic fields. To avoid non-essential technical complications, we do not investigate multivariate processes although it is not hard to see that part of the results can be generalized from the univariate to the multivariate cases in a straightforward way. 

In Section 3 we give precise formulations of the stability problems we wish to discuss, and we try to establish interrelations between them. We prove general stability results, among them amplifications of Theorems \ref{t1.1} and \ref{t1.2}, cf.~Theorem 3.10. The reminder of the paper is devoted to particular prediction problems. 

In Section 4 we consider stability of interpolation of one missing value, which is one of the simplest prediction problems and was already studied by A.~N.~Kolmogorov in \cite{[16],[17]}. Corollary 4.10 exhibits that for the interpolation problem Theorems \ref{t1.1} and \ref{t1.2} remain to be valid if convergence in norm is replaced by a much weaker form of convergence, cf.~Definition 2.5. Remarkably, it seems not so easy to prove stability results for the interpolation problem in case if more than one value is missing. 

Section 5 deals with the $m$-steps-ahead prediction problem, $m \in \mathbb N$, of a sequence of random variables. For $m=1$ it is closely related to Szeg\"o's infimum problem, \cite{[29]}. Although the $m$-steps-ahead prediction problem probably is the most important and most extensively studied prediction problem, several questions remain unanswered. For example, at present a complete extension of the results in case $m=1$ to arbitrary values of $m \in \mathbb N$ is not known. 

In Section 6 we briefly discuss the case in which the set $S$ of observation points is finite. Since the approximating linear subspace is finite dimensional, rather strong stability results can be obtained, cf.~Theorem 6.2. However, to prove them we had to suppose that the dimension of the approximating subspace is equal to the number of elements of the set $S$. It would be interesting to explore what happens if this dimension is less than the cardinality of $S$. In any case the results of Section 6 suggest that from point of view of stability it is reasonable to compute the prediction of a discete time series on the basis of a finite observation set.

The final Section 7 is devoted to a periodic set of observation points, i.e. $S$ is a translate of a closed subgroup $H$ of the parameter group. Problems of this type are related to the famous Whittaker-Shannon-Kotel'nikov theorem which is of great importance in information theory. Although progress has been made in recent years, cf.~\cite{[22]}, almost all known facts pertain the stationary or Hilbert space case $p=2$. As Theorem 7.8 shows, under the assumption that the annihilator group of $H$ is at most countable, the prediction has a rather nice limit behaviour. The reason is that the prediction is bounded by $1$ independently of the underlying spectral measure.

 Along with affirmative claims there exist many negative assertions, and examples and counterexamples constitute an essential part of the present paper. To point out the variety of results let us mention some of them. It is not surprising that there is a strong dependency on $p$, in general. It is more remarkable that according to Szeg\"o's infimum theorem the one-step-ahead prediction error does not depend on $p \in (0,\infty)$. For $m$-steps-ahead prediction, the independence on $p \in [1,\infty)$ can be proved, cf.~Corollary 5.12. Clearly, the results may depend on the mode of convergence and on the observation set $S$. But, there is also an interplay between them, compare Theorem 4.6 and Corollaries 4.7 and 4.8 with Example 6.4. There may also occure a dependency on the direction of convergence. Theorem \ref{t1.1} and Corollaries 3.3(ii) and 5.7 seem to indicate that convergence from above provides better stability properties than convergence from below. Thus, if one has to choose a spectral density among several candidates, reasonably one may choose that, whose minimum modulus is the largest. It is also worth mentioning that strong stability results can be obtained as soon as the observations at the points of $S$ give full information (in the linear sense) on the underlying process, cf.~Corollary 3.3(i), Proposition 3.8, and the end of Section 3. 

%=============================================%

\section{Modes of Convergence}

Let $\Gamma$ be a locally compact topological space with Hausdorff topology. For a subset $B \subseteq \Gamma$, denote by $1_B$ its indicator function and set $1 := 1_\Gamma$. Let $\mathcal B$ be the $\sigma$-algebra of Borel subsets of $\Gamma$ and $\mathcal M$ be the set of all non-zero finite non-negative regular measures on $B$. If $\nu$ is a $\mathbb C$-valued measure on $B$, denote its total variation by $\|\nu\|$. Recall, that $\|\cdot\|$ is a norm on the space of all $\mathbb C$-valued Borel measures on $\Gamma$. If $\mu \in \mathcal M$ and $\nu$ is absolutely continuous with respect to $\mu$, we write $\nu << \mu$ and denote by $\frac{d\nu}{d\mu}$ a Radon-Nikodym derivative. Recall, that
\begin{equation} \label{e2.1} 
    \|\nu\|=\int \left| \frac{d\nu}{d\mu} \right| \, d\mu \, ,
\end{equation}
where here and henceforth integration is over $\Gamma$ in case the domain of integration is not indicated. For $p \in (0,\infty]$, let $L^p(\mu)$ denote the common $L^p$ space of ($\mu$-equivalence classes of) $\mathbb C$-valued functions related to $\mu$. 

Throughout the present paper, let $\mu_0 \in \mathcal M$ and let $(\mu_n)_{n \in \mathbb N}$ (in short, $(\mu_n)$) denote a sequence of measures of $\mathcal M$. If $\nu \in \mathcal M$ is such that $\mu_k << \nu$, set $w_k:= \frac{d\mu_k}{d\nu}$, $k\in {\mathbb N}_0$. As a consequence of (\ref{e2.1}) the equalities $\lim \|\mu_0-\mu_n \|=0$ and $\int |w_0-w_n| \, d\nu =0$ are equivalent. Moreover, one easily concludes that $\lim \|\mu_0-\mu_n \|=0$ if and only if $\lim \mu_n(B)=\mu_0(B)$ uniformly for all $B \in \mathcal B$. 

It is near at hand to discuss certain modes of weak convergence along with norm convergence.

\begin{definition} \label{d2.1}
  We say that $(\mu_n)$ converges to $\mu_0$ weakly and write $\rm w$-$\lim \mu_n =\mu_0$ if the sequence $(w_n)$ tends to $w_0$ with respect to the weak topology of $L^1(\nu)$, i.e. if 
  \begin{equation} \label{e2.2}
     \lim \int f w_n \, d\nu = \int f w_0 \, d\nu
  \end{equation}
  for all $f \in L^\infty(\nu)$. 
\end{definition}

The following lemma gives a characterization of weak convergence and unveils that Definition \ref{d2.1} does not depend on the choice of the measure $\nu \in \mathcal M$.

\begin{lemma} \label{l2.2}
   The sequence $(\mu_n)$ converges weakly to $\mu_0$ if and only if
   \begin{equation} \label{e2.3}
      \lim \mu_n(B) =\mu_0(B)
   \end{equation}
   for all $B \in \mathcal B$.
\end{lemma}

\begin{proof}
Choosing $f=1_B$ in (\ref{e2.2}) we get (\ref{e2.3}), in particular, $\lim \| \mu_n \| =  \lim \mu_n(\Gamma) = \mu_0(\Gamma)=\| \mu_0 \|$. So the sequence $(w_n)$ is bounded in $L^1(\nu)$. Since the linear space of simple functions is dense in $L^\infty(\nu)$, (\ref{e2.2}) follows from (\ref{e2.3}).
\end{proof}

\begin{definition} \label{d2.3}
  We say that $(\mu_n)$ converges to $\mu_0$ in the weak* sense and write $\rm w^*$-$\lim \mu_n= \mu_0$ if $\lim  \int f \, d\mu_n = \int f \, d\mu_0$ for any bounded continuous $\mathbb C$-valued function on $\Gamma$. 
\end{definition}

Note, that $\rm w^*$-$\lim \mu_n =\mu_0$ yields $\lim \mu_n(\Gamma) =\mu_0(\Gamma)$. A sufficient condition for weak* convergence is $\lim \mu_n(U) =\mu_0(U)$ for those open subsets $U$ of $\Gamma$ whose boundary has $\mu_0$-measure $0$. If $\Gamma$ is a normal topological space, the condition is also necessary, cf.~\cite[Thm. 4.5.15]{[5]}.

\begin{remark} \label{r2.4}
In probability theory the mode of convergence introduced in Definition \ref{d2.3} is often called ''weak convergence''. Since $\mathcal M$ is a subset of the dual space of the Banach space of bounded continuous functions the differing notion is in accordance with the terminology of functional analysis. 
\end{remark}

Let $f_k$, $k \in {\mathbb N}_0$, be $\mathbb C$-valued measurable functions on $\Gamma$ and let $\nu \in\mathcal M$. Recall, that a sequence $(f_n)$ converges to $f_0$ in measure $\nu$ (notation $\nu$-$\lim f_n = f_0$) if for arbitrary $\varepsilon \in (0,\infty)$ there exists a $n_0 \in \mathbb N$ such that $\nu(|f_0-f_n| > \varepsilon) < \varepsilon$ for any $n \geq n_0$ in $\mathbb N$. Since in practice the estimated spectral measure of the modelling stationary process is often absolutely continuous with respect to the Lebesgue measure $\lambda$ on the interval $[0, 2\pi)$, cf. \cite{[2]},  one could also study convergence in measure $\lambda$. To apply the concept of convergence in measure to general sequences of bounded measures we give the following definition.

\begin{definition} \label{d2.5}
   We say that $(\mu_n)$ converges to $\mu_0$ in measure and write m-$\lim \mu_n = \mu_0$ if $(w_n)$ tends to $w_0$ in measure $\nu$. 
\end{definition}

This definition can only be correct if the coined notion does not depend on the choice of $\nu$. 

\begin{lemma} \label{l2.6}
   Let $\nu_1, \nu_2 \in \mathcal M$ be two measures with $\mu_k << \nu_1$ and $\mu_k << \nu_2$. Set $w^{(i)}_k := \frac{d\mu_k}{d\nu_i}$ for $k \in {\mathbb N}_0$, $i \in \{ 1,2\}$. Then $\nu_1$-$\lim w^{(1)}_n = w^{(1)}_0$ if and only if $\nu_2$-$\lim w^{(2)}_n = w^{(2)}_0$.
\end{lemma}

\begin{proof}
First, suppose $\nu_1 << \nu_2$ and set $v := \frac{d\nu_1}{d\nu_2}$. Recall, that $w^{(2)}_k = w^{(1)} v$ by the chain rule for Radon-Nikodym derivatives. Let $\varepsilon \in (0,\infty)$. Since $v \in L^1(\nu_2)$, there exists a number $c \in (0,\infty)$ satisfying $\nu_2( v > c)< \frac{\varepsilon}{2}$. If $\nu_1$-$\lim w^{(1)}_n = w^{(1)}_0$, we can choose a number $n_0 \in \mathbb N$ such that $\nu_1(|w^{(1)}_0 - w^{(1)}_n| > \frac{\varepsilon}{c}) < \frac{\varepsilon}{2c}$ for all numbers $n \geq n_0$.  Consequently, 
\begin{eqnarray*}
    \nu_2(|w^{(2)}_0 - w^{(2)}_n| > \varepsilon) & \leq & \nu_2((|w^{(1)}_0 - w^{(1)}_n| > \frac{\varepsilon}{c}) \cap (v \leq c)) + \nu_2(v > c) \\
                                                                         &  <   &  c \cdot \nu_1(|w^{(1)}_n-w^{(1)}_0| > \frac{\varepsilon}{c}) + \frac{\varepsilon}{2} \\
                                                                         & \leq & \varepsilon
\end{eqnarray*}
for all numbers $n \geq n_0$. 

Conversely, let $\nu_2$-$\lim w^{(2)}_n = w^{(2)}_0$. To prove the equality $\nu_1$-$\lim w^{(1)}_n = w^{(1)}_0$ it is sufficient to show it on the set $B := \{\gamma \in \Gamma : v(\gamma) > 0 \}$. Since the measures $\nu_1$ and $\nu_2$ are equivalent on $B$, the assertion follows from the first part of the proof. Finally, if $\nu_1$ and $\nu_2$ are arbitrary, we have $\nu_1$-$\lim w^{(1)}_n = w^{(1)}_0$ if and only if $(\nu_1+\nu_2)$-$\lim w^{(1)}_n \frac{d\nu_1}{d(\nu_1+\nu_2)} = w^{(1)}_0 \frac{d\nu_1}{d(\nu_1+\nu_2)}$, if and only if $(\nu_1+\nu_2)$-$\lim w^{(2)}_n \frac{d\nu_2}{d(\nu_1+\nu_2)} = w^{(2)}_0 \frac{d\nu_2}{d(\nu_1+\nu_2)}$, if and only if $\nu_2$-$\lim w^{(2)}_n = w^{(2)}_0$.
\end{proof}

Note that Lemma 2.6 fails if the measures $\nu_1$ and $\nu_2$ are allowed to be $\sigma$-finite. For example, if $\Gamma := \mathbb N, \mu_n := \delta_n$ is the Dirac measure at the point $n \in \mathbb N$, $\nu_1 :=\sum_{n=1}^\infty \frac{1}{n^2}\delta_n$ and $\nu_2$ is the countimg measure, then $\nu_1$-$\lim w_n^{(1)} = 0$ whereas $\nu_2$-$\lim w_n^{(2)}$ does not exist.

 Of course, all these convergence notions introduced above are equivalent if $\Gamma$ is a finite set. In general, weak convergence does not imply convergence in measure, which in turn does not yield weak* convergence. Norm convergence is stronger than weak  convergence, as well as it is stronger than convergence in measure, and weak convergence is stronger than weak* convergence. In general, these inclusions can be sharp. However, there exist interesting particular cases, in which some of these notions are equivalent. For example, if $\Gamma= [0,1]$ and the measures $\mu_k$, $k \in {\mathbb N}_0$, are absolutely continuous with respect to the Lebesque measure $\lambda$ then weak and weak* convergence coincide. To see this, note that $\Gamma=[0,1]$ is normal and compact as a topological space. Furthermore, any open subset $U$ of $\Gamma$ is an at most countable union of open intervals, which implies that the boundary of $U$ has $\mu_0$-measure $0$. Therefore, $\rm w^*$-$\lim \mu_n = \mu_0$ if and only if $\lim\mu_n(U) = \mu_0(U)$ for any open subset $U$ of $\Gamma$. According to a theorem by Jean Dieudonn\'e, cf.~\cite[IV.16]{[5]}, this is equivalent to the condition $\lim \mu_n(B) = \mu_0(B)$ for any $B \in \mathcal B$. So $\rm w$-$\lim\mu_n =\mu_0$ by Lemma 2.2. 

\begin{proposition} \label{p2.7}
  Let $\Gamma$ be a discrete space, i.e. all subsets of $\Gamma$ are open. Then weak* convergence implies convergence in norm. If $\,\Gamma$ is an infinite set, there exists a sequence converging in measure, but not converging in the weak* sense. 
\end{proposition}

\begin{proof}
We can assume that $\Gamma$ is infinite since the finite case is known. Select a countable subset $D := \{ \gamma_i :i \in {\mathbb N}\}$ such that $\mu_k(\Gamma \setminus D)=0$ for all $k \in {\mathbb N}_0$. Identify $\mu_k$ with the sequence $s_k := (\mu_k(\{ \gamma_i \}))_{i \in \mathbb N}$. The sequence $s_k$ is an element of the space $l_1$ of absolutely summable sequences, and $\lim \|\mu_0-\mu_k\|=0$ is equivalent to $\lim s_k = s_0$ with respect to the norm topology of $l_1$. Similarly, $w^*$-$\lim \mu_n = \mu_0$ is equivalent to $\lim s_n = s_0$ with respect to the weak topology of $l_1$,  because the space $l_\infty$ and the space of bounded continuous functions on $\mathbb N$ coincide.  Therefore, the first assertion is a consequence of the fact that on the space $l_1$ norm convergence and weak convergence of sequences are equivalent, cf.~\cite[pp.~218-220]{[23]}. To prove the second assertion, set $\mu_n(\{\gamma_i\}) := \delta_{n,i}$ for any $n,i \in {\mathbb N}$. We obtain $\rm m$-$\lim \mu_n =0$. Moreover, $\lim \int 1_D \, d\mu_n = 1$ and $\lim \int 1_{\{ \gamma \}} \, d\mu_n =0$ for any $\gamma \in \Gamma$, what implies the non-existence of a weak* limit of the sequence $(\mu_n)$. 
\end{proof}

It is worth mentioning that there exist sequences $(\mu_n)$ on $\Gamma =[0,1]$ such that both the limits $\rm w^*$-$\lim \mu_n$ and $\rm m$-$\lim \mu_n$ exist, but are unequal. To give an example, set $d\mu_n:= n \cdot 1_{[0,\frac{1}{n})} \, d\lambda$ on $[0,1]$. This gives $\rm m$-$\lim\mu_n =0$ and $\rm w^*$-$\lim \mu_n = \delta_0$, the Dirac measure at zero. 

\begin{proposition} \label{p2.8}
  If $\rm w$-$\lim\mu_n$ and $\rm m$-$\lim \mu_n$ exist, they are equal. 
\end{proposition}

\begin{proof}
Let $\rm w$-$\lim \mu_n=\mu^{(1)}_0$ and $\rm m$-$\lim \mu_n = \mu^{(2)}_0$. Let $\nu \in \mathcal M$ be such that $\mu_n << \nu$, $\mu^{(1)}_0, \mu^{(2)}_0 << \nu$. Set $w^{(i)}_0=\frac{d\mu^{(i)}_0}{d \nu}$, $i \in \{ 1,2\}$. If $\mu^{(1)}_0 \not= \mu^{(2)}_0$ there exist a set $B \in \mathcal B$ with $\nu(B) > 0$ and 
\begin{equation} \label{e2.4}
   w^{(1)}_0 \not= w^{(2)}_0 \qquad \nu {\rm - a.e.} \,\, {\rm on} \,\, B \, .
\end{equation}
Choosing a subsequence $(\mu_m)$ of $(\mu_n)$ with $\lim_{m \to \infty} w_m = w^{(2)}_0$ $\nu$-a.e. and applying Egorov's theorem, we find a set $C \subseteq B$ satisfying $\nu(C)>0$ and
\begin{equation} \label{e2.5}
  \lim_{m \to\infty} 1_Cw_m = 1_Cw^{(2)}_0
\end{equation}
uniformly. If $\mu(C \cap \cdot)$ denotes the restriction of a measure $\mu \in \mathcal M$ to ${\mathcal B}(C)$, we can write $\rm w$-$\lim_{m \to\infty} \mu_m(C \cap \cdot)=\mu^{(2)}_0(C \cap \cdot)$ by (\ref{e2.5}) and $\rm w$-$\lim_{m \to\infty} \mu_m(C \cap \cdot) = \mu^{(1)}_0(C \cap \cdot)$ by Lemma \ref{l2.2}. Therefore, $\mu^{(1)}_0(C \cap \cdot) = \mu^{(2)}_0(C \cap \cdot)$, what implies $w^{(1)}_0=w^{(2)}_0$ $\nu$-a.e. on $C$, a contradiction to (\ref{e2.4}).
\end{proof}

In \cite[Lemma 1 of Ch.~1]{[6]} it was shown that for $\mu \in \mathcal M$, $p \in (0, \infty)$, $f_k \in L^p(\mu)$ with $k\in {\mathbb N}_0$ the conditions
\begin{equation} \label{e2.6}
   \lim\int | f_n |^p \, d\mu =\int |f_0|^p \, d\mu \, 
\end{equation}
and 
\begin{equation} \label{e2.7}
   \lim f_n = f_0 \qquad \mu {\rm -a.e.} \,
\end{equation}
 imply
\begin{equation} \label{e2.8}
  \lim \int |f_0-f_n|^p \, d\mu = 0 \, . 
\end{equation}

For future use we state a slight generalization of this assertion.

\begin{lemma} \label{l2.9}
  From (\ref{e2.6}) and $\mu-\lim f_n = f_0$ the limit equation (\ref{e2.8}) follows. 
\end{lemma}

\begin{proof}
Assume the contrary, i.e. there exists $c \in (0,\infty)$ and a subsequence $(f_m)$ of the sequence $(f_n)$ such that 
\begin{equation} \label{e2.9}
  \int | f_0-f_m |^p \, d\mu \geq c
\end{equation}
for any $m$. Again, choose a subsequence $(f_l)$ of the sequence $(f_m)$ such that $\lim f_l = f_0$ $\mu$-a.e.. From the result cited above one derives $\lim_{l \to \infty} \int | f_0-f_l |^p \, d\mu = 0$, what contradicts (\ref{e2.9}). 
\end{proof}

Applying Lemma \ref{l2.9} with $p=1$ and $f_k=w_k$ we get the following result.

\begin{corollary} \label{c2.10}
  If $\rm w^*$-$\lim \mu_n = \rm m$-$\lim \mu_n =\mu_0$, then $\lim \| \mu_0-\mu_n \| = 0$. 
\end{corollary}

%==============%

\section{General results}

Let $G$ be an LCA group, i.e. a locally compact abelian group with Hausdorff topology. Let $\Gamma$ denote its dual group of continuous characters. The group $G$ can be identified with the dual of $\Gamma$, and the character on $\Gamma$ generated by $x \in G$ is denoted by $e_x$. Group operations are written additively, and the letter $\lambda$ stands for a Haar measure on $\Gamma$, which in case of compact $\Gamma$ is presumed to be normalized, i.e. $\lambda(\Gamma)=1$. Recall, that $\mathcal M$ denotes the set of all non-zero non-negative regular measures on the Borel $\sigma$-algebra $\mathcal B$ of $\Gamma$. 

Let $S$ be a non-empty subset of $G \setminus \{ 0 \}$ and ${\mathcal T}(S)$ be the linear space of all complex-valued trigonometric $S$-polynomials, i.e. the linear space of all finite sums of the form $\Sigma_j a_je_{x_j}$ with $a_j \in {\mathbb C}$, $x_j \in S$. If $p \in (0, \infty)$ and $\mu \in {\mathcal M}$, define the distance of the functrion $1$ to ${\mathcal T}(S)$ with respect to the metric of $L^p(\mu)$ by
\begin{equation} \label{(e3.1)}
    d_p(\mu ; S):= \inf \left\{ \int |1-\tau|^p \, d\mu : \tau \in {\mathcal T}(S) \right\} \, .
\end{equation}
For convenience, $d_p(\mu ; S)$ is called the distance of $1$ to ${\mathcal T}(S)$. Although if $p \in (1, \infty )$, it would be more correct to call it the $p$-th power of the distance. 

Let $[{\mathcal T}(S)]_{\mu, p}$ be the closure of ${\mathcal T}(S)$ in $L^p(\mu)$. If there exists a unique element $\phi_{\mu, p} \in [{\mathcal T}(S)]_{\mu, p}$ satisfying $d_p(\mu ; S) = \int |1-\phi_{\mu,p}(S)|^p \, d\mu$ it is called the metric projection of $1$ onto $[{\mathcal T}(S)]_{\mu,p}$. Recall, that in case $p \in (1, \infty )$ the metric projection exists for all $S$. 

Let $\mu \in {\mathcal M}$ be such that $\mu << \nu$. Then the $\nu$-equivalence yields $\mu$-equivalence, and if $\phi_{\nu,p}(S)$ exists the integral
\begin{equation} \label{(e3.2)}
    d_p(\nu,\mu ; S) := \int |1-\phi_{\nu,p}|^p \, d\mu
\end{equation}
is defined uniquely. Note, that values like $d_p(\nu,\mu ; S)= \infty$ in case $\phi_{\nu,p}(S) \not\in L^p(\mu)$. If $\mu$ is not absolutely continuous with respect to $\nu$, there exists a $B \in {\mathcal B}$ with $\mu(B) >0$ and $\nu(B)=0$. Since $\phi_{\nu,p}(S)$ can be arbitrarily chosen on $B$, the integral on the right hand side of (\ref{(e3.2)}) cannot be given a sense. Therefore, whenever we shall be concerned with $d_p(\nu,\mu ; S)$ we always suppose that $\phi_{\nu,p}(S)$ exists and $\mu << \nu$ although this will not be mentioned explicitely each time.  If $\nu$ is the estimate of the spectral measure $\mu$ of a harmonizable symmetric $p$-stable process, then $d_p(\nu,\mu ; S)$ can be interpreted as the true error of the estimated prediction $\phi_{\nu,p}(S)$. We should conclude that $\phi_{\nu,p}(S)$ is an unsuitable prediction in case it does not belong to $L^p(\mu)$. 

Let $\mu_k \in {\mathcal M}$, $k \in {\mathbb N}_0$. As described in the introduction, from the point of view of prediction theory it is of interest to study the behaviour of the sequence $(\phi_{\mu_n,p})_{n \in {\mathbb N}}$ if $(\mu_n)$ tends to $\mu_0$ in one or another way. More precisely, we try to give necessary and sufficient conditions for any of the following relations:
\begin{itemize}
\item[(R1)] $\lim \, d_p(\mu_n ; S) = d_p(\mu_0 ; S)$ , 
\item[(R2)] $\lim \, d_p(\mu_n,\mu_0 ; S) = d_p(\mu_0 ; S)$ ,
\item[(R3)] $\lim \, \int |\phi_{\mu_0,p}-\phi_{\mu_n,p} |^p \, d\mu_0 =0$ .
\end{itemize}
Of course, in the cases (R2) and (R3) we presume that the respective metric projections exist and $\mu_0 << \mu_n$ for any $n \in {\mathbb N}$. To simplify the notation we set $\phi_{k,p}(S) := \phi_{\mu_k,p}(S)$, and frequently we shall not indicate the dependency on the set $S$, e.g. ${\mathcal T}:={\mathcal T}(S)$, $d_p(\mu_k):=d_p(\mu_k ; S)$, $\phi_{k,p}:= \phi_{k,p}(S)$ etc. .

Our first result is a generalization of Theorem \ref{t1.2}.

\begin{theorem} \label{t3.1}
If ${\rm w}^*$-$\lim \mu_n =\mu_0$, then 
\begin{equation} \label{(e3.3)}
   \lim \sup d_p(\mu_n ; S) \leq d_p(\mu_0 ; S)
\end{equation}
for all $S$. 
\end{theorem}

\begin{proof}
For arbitrarily fixed $\varepsilon \in (0, \infty )$ choose $\tau_\varepsilon \in{\mathcal T}$ such that $\int |1-\tau_\varepsilon |^p \, d\mu_0 < d_p(\mu_0) + \frac{\varepsilon}{2}$ . Since the function $|1-\tau_\varepsilon|^p$ is continuous and bounded, there exists $n_0 \in {\mathbb N}$ such that $\left| \int |1-\tau_\varepsilon|^p \, d\mu_0 - \int |1-\tau_\varepsilon|^p \, d\mu_n \right| < \frac{\varepsilon}{2}$ for any $n \geq n_0$. Hence, 
\begin{eqnarray*}
d_p(\mu_n) & \leq & \int |1-\tau_\varepsilon|^p \, d\mu_n \\
                  & \leq  & \int |1-\tau_\varepsilon|^p \, d\mu_0 + \left| \int |1-\tau_\varepsilon|^p \, d\mu_0 - \int |1-\tau_\varepsilon|^p \, d\mu_n \right| \\
                  & < & d_p(\mu_0)+ \varepsilon
\end{eqnarray*} 
for any $n \geq n_0$, which yields (\ref{(e3.3)}) by the arbitrarity of $\varepsilon > 0$. 
\end{proof}

To see that inequality (\ref{(e3.3)})  can be sharp consider the interpolation problem for weakly stationary sequences where $0$ is the only missing value. Thus, let $G=\mathbb Z$ and $S={\mathbb Z} \setminus \{ 0 \}$ for the moment. Defining $e_x(\gamma) := {\rm e}^{ix\gamma}$, $x \in \mathbb Z$, $\gamma \in [0, 2 \pi)$, the dual group $\Gamma$ of $\mathbb Z$ can be identified with the interval $[0,2 \pi )$, where the group operation is addition modulo $2 \pi$ and the set of all open subintervals together with all sets of the form $[0,a) \cup (b,2 \pi )$ form a basis of its topology. Let $\lambda$ be the normalized Lebesgue measure on $[0,2 \pi )$. Assume, that $\mu_k << \lambda$ and denote $w_k:=\frac{d\mu_k}{d\lambda}$, $k \in {\mathbb N}_0$. A well-known fact is that $d_2(\mu_k)= (\int w^{-1}_k \, d\lambda)^{-1}$ in case $w^{-1}_k \in L^1(\lambda)$, and $d_2(\mu_k)=0$ otherwise, in particular, in case when $w_k=0$ on a set of positive $\lambda$-measure, \cite[Thm. 3]{[17]}. 

\begin{example} \label{e3.2}
a) $\,$ If $w_0=1$, $w_n=1_{[\frac{1}{n},2 \pi )}$, then $\lim \| \mu_0-\mu_n\|=0$, $d_2(\mu_0)=1$ and $d_2(\mu_n)=0$. 

b) Let $w_0(\gamma)=\sqrt{\gamma}$ for $\gamma \in [0,2 \pi )$. For $a \in (0, \infty )$ set $w_n=\frac{a}{n} \cdot 1_{[0,\frac{1}{n})} + w_0 \cdot 1_{[\frac{1}{n}, 2 \pi )}$. So $\lim \| \mu_0-\mu_n \|=0$ and $d_2(\mu_0)=\frac{2 \pi}{2 \sqrt{2 \pi}}$, $d_2(\mu_n)=2 \pi [\frac{1}{a} + 2(\sqrt{2 \pi}-\sqrt{\frac{1}{n}})]^{-1}$, and hence, $\lim d_2(\mu_n)=2 \pi (\frac{1}{a}+2\sqrt{ 2 \pi})^{-1} < d_2(\mu_0)$. If $w_n=\frac{1}{n^2} \cdot 1_{[0,\frac{1}{n})} +w_0 \cdot 1_{[\frac{1}{n},2 \pi )}$, then $d_2(\mu_n)=2 \pi [n+2(\sqrt{2 \pi}-\sqrt{\frac{1}{n}})]^{-1}$, and consequently, $\lim d_2(\mu_n)=0$. Therefore, for any $b \in [0,d_2(\mu_0)]$ there exists a sequence $(\mu_n)$ with $\lim \|\mu_0-\mu_n \|=0$ and $\lim d_2(\mu_n)=b$. Note, that $(w_n)$ tends to $w_0$ even uniformly and that, additionally, $\lim \int \log \frac{w_0}{\min (w_0,w_n)} \, d\lambda = 0$. Compare this with Theorem \ref{t5.5}. 
\end{example}

The following corollary is a straightforward consequence of Theorem \ref{t3.1}. Its first assertion claims that as soon as the observations at the points of $S$ contain the whole (linear) information on the process, then the computed prediction error tends to be equal to the theoretical prediction error.

\begin{corollary} \label{c3.3}
Let $\rm w^*$-$\lim \mu_n = \mu_0$. 
\begin{itemize}
\item[(i)] $\,$ If $d_p(\mu_n ; S) = 0$, then (R1) is true. 
\item[(ii)] $\,$ If there exists a sequence $(c_n)$ of positive real numbers such that $\lim \inf c_n \geq 1$ and $c_n\mu_0 \leq \mu_n$, $n \in {\mathbb N}$,
 then (R1) is satisfied. 
\end{itemize}
\end{corollary}

\begin{corollary} \label{c3.4}
Let $S \subseteq G \setminus \{ 0 \}$ be an arbitrary non-empty set and $\mu_0 \in \mathcal M$. The following two assertions are equivalent: 
\begin{itemize}
\item[(i)] For any sequence $(\mu_n)$ with $\lim \|\mu_0-\mu_n \|=0$, relation (R1) is true. 
\item[(ii)] For any sequence $(\mu_n)$ such that $\lim \|\mu_0-\mu_n \|=0$ and $\mu_n \leq \mu_0$, $n \in \mathbb N$, the inequality $d_p(\mu_0; S) \leq \lim \inf d_p(\mu_n ; S)$ is true.
\end{itemize}
\end{corollary}
 
\begin{proof}
(i) $\to$ (ii) is clear. To prove the opposite conclusion (ii) $\to$ (i) choose $\nu \in \mathcal M$ such that $\mu_k << \nu$ and define $d\mu'_k:= \min \left( \frac{d\mu_0}{d\nu} , \frac{d\mu_k}{d\nu} \right) \, d\nu$, $k \in {\mathbb N}_0$. Since (ii) implies that $\lim \inf d_p(\mu_k,S) \geq \lim \inf d_p(\mu'_k,S) \geq d_p(\mu_0,S)$, assertion (i) follows from (\ref{(e3.3)}).
\end{proof}

Let $d_\infty(\mu ; S) := \inf \{ {\esssup}_\mu |1-\tau| : \tau \in {\mathcal T}(S) \}$, where ${\esssup}_\mu$ denotes the essential supremum with respect to the measure $\mu \in \mathcal M$. For the sake of completeness we show by example that for $p=\infty$ the inequality (\ref{(e3.3)}) is not true, in general. 

\begin{example} \label{e3.5}
Let $G$, $\Gamma$ and $\lambda$ be as in Example (\ref{e3.2}). Let $S$ be the singleton $\{ 1 \}$, $\mu_0:=\delta_0$, $\mu_n := \delta_0+\frac{1}{n} \lambda$. Clearly, $\lim \|\mu_0-\mu_n \|=0$ and $d_\infty (\mu_0)=\inf \{ |1-\alpha| : \alpha \in {\mathbb C} \}=0$. Moreover, 
\begin{eqnarray*}
d_\infty (\mu_n) & = & \inf \{ {\esssup}_\lambda \, |1-\alpha e_1| : \alpha \in {\mathbb C} \} \\
                        & = & \inf \{ \max \{ (1+|\alpha|^2 - 2\, {\rm Re}(\alpha {\rm e}^{i \gamma} ))^{\frac{1}{2}} : \gamma \in [0,2 \pi ) \} : \alpha \in {\mathbb C} \} \\
                        & = & \inf \{\max \{ |(1+|\alpha|^2 + 2|\alpha|)^{\frac{1}{2}} : \alpha \in {\mathbb C} \} \\
                        & = & 1 
\end{eqnarray*}
for any $n \in \mathbb N$. 
\end{example}

Our next example gives an impression to what extend $\lim \, d_p(\mu_n,\mu_0 ; S)$ can differ from $d_p( \mu_0 ; S)$ if $\lim \|\mu_0-\mu_n \| =0$

\begin{example} \label{e3.6}
Let $G$, $\Gamma$ and $\lambda$ be as in Example \ref{e3.2}. Since for $d\mu= w \, d\lambda$ one has $\phi_{\mu,2}=1-(\int w^{-1} \, d\lambda)^{-1}w^{-1}$ if $w^{-1} \in L^1(\lambda)$ and $\phi_{\mu,2}(S)=1$ otherwise, the following results can be obtained by simple computations:
\newline
a) $\,$ If $w_0=1$ and $w_n(\gamma) = \gamma \cdot 1_{[0,\frac{1}{n})}(\gamma) + 1_{[\frac{1}{n},2 \pi )}(\gamma)$, $\gamma \in [0,2 \pi)$, then $d_2(\mu_0)=1$, $\phi_{n,2}=1$ and $d_2(\mu_n,\mu_0)=0$.
\newline
b) $\,$ If $w_0(\gamma)=\gamma$ and $w_n(\gamma)= \frac{1}{n} \cdot 1_{[0, \frac{1}{n})}(\gamma) + \gamma \cdot 1_{[\frac{1}{n}, 2 \pi - \frac{1}{n})}(\gamma) + (2\pi - \gamma) \cdot 1_{[2\pi -\frac{1}{n}, 2\pi)}(\gamma)$, $\gamma \in [0,2\pi )$, then $d_2(\mu_0)=0$, $\phi_{n,2} \not\in L^2(\mu_0)$, and hence, $d_2(\mu_n,\mu_0)=\infty$.
\newline
c) $\,$ Let $a \in (0,\infty)$, $w_0(\gamma)=a\gamma$ for $\gamma \in [0,2\pi )$, $w_n(\gamma)= \frac{a}{n^2} \cdot 1_{[0,\frac{1}{n})} + w_0 \cdot 1_{[\frac{1}{n},2 \pi )}$. We get $d_2(\mu_0)=0$, $d_2(\mu_n,\mu_0)= 2 \pi a (\frac {n^2}{2} + \log 2 \pi + \log n)(n + \log 2 \pi + \log n)^{-2}$, what tends to $\pi a$ if $n \to \infty$. Thus, for any $b \in [0,\infty ]$ there exists a measure  $\mu_0$ and a sequence of measures $(\mu_n)$ with $\lim \| \mu_0-\mu_n \|=0$, $d_2(\mu_0)=0$ and $\lim d_2(\mu_n,\mu_0)=b$. 
\newline
d) $\,$ If $a \in (0,\infty)$, $w_0(\gamma)=\sqrt{\gamma}$ for $\gamma \in [0,2 \pi )$, $w_n(\gamma)= a n^{-\frac{3}{4}} \cdot 1_{[0,\frac{1}{n})} + w_0 \cdot 1_{[\frac{1}{n}, 2 \pi )}$, then $d_2(\mu_0)=\frac{\sqrt{2\pi}}{2}$, $d_2(\mu_n,\mu_0)= \left[\frac{1}{2\pi}(\frac{1}{3a^2} + 2(\sqrt{2 \pi}-\sqrt{\frac{1}{n}}))\right] \cdot \left[\frac{1}{2 \pi}(\frac{1}{4n^{\frac{1}{4}}} + 2(\sqrt{2\pi}-\sqrt{\frac{1}{n}}))\right]^{-2}$, what tends to $\frac{1}{2}(\frac{1}{3a^2} + 2 \sqrt{2\pi})$ if $n \to \infty$. Consequently, for any $b \in [d_2(\mu_0), \infty )$ there exists a sequence $(\mu_n)$ with $\lim \|\mu_0-\mu_n\|=0$ and $\lim d_2(\mu_n,\mu_0)=b$. 
\end{example}

Let us discuss relations between (R1), (R2) and (R3). Obviously, (R3) yields (R2) without any convergence condition on the sequence $(\mu_n)$. The following assertion is a straightforward consequence of Lemma \ref{l2.9}.

\begin{proposition} \label{p3.7}
   If $\mu_0$-$\lim \phi_{n,p}(S) = \phi_{0,p}(S)$, then (R2) yields (R3). 
\end{proposition}

We give further sufficient conditions for the implication (R2) $\to$ (R3) that are closer to prediction theory in spirit. Proposition \ref{p3.8} tells us that (R3) follows from (R2) if the observations at the points of $S$ contain the whole information on the process. The proof of Theorem \ref{t3.9} is based on the uniform rotundity of the $L^p(\mu)$-norm, $p \in (1,\infty)$. 

\begin{proposition} \label{p3.8}
   If $d_p(\mu_0 ; S)=0$ then (R2) yields (R3). 
\end{proposition}

\begin{proof}
If (R2) is satisfied then $\lim \int |1-\phi_{n,p}|^p \, d\mu_0 = d_p(\mu_0) = 0$, hence, $\lim \phi_{n,p}=1=\phi_{0,p}$ with respect to the metric of $L^p(\mu_0)$.
\end{proof}

\begin{theorem}  \label{t3.9}
   Let $p \in (1,\infty)$. If $\phi_{n,p}(S) \in [{\mathcal T}(S)]_{\mu_0,p}$ for all $n \in {\mathbb N}$, then (R3) follows from (R2).
\end{theorem}

\begin{proof}
Let $p \in (1,\infty)$. For $f \in L^p(\mu_0)$ define $\|f\|_0 := \left(\int |f|^p \, d\mu_0 \right)^\frac{1}{p}$. Recall that $\| \cdot \|_0$ is a uniformly rotund norm, cf.~\cite[p. 441]{[23]}. By Proposition \ref{p3.8} we can assume that $d_p(\mu_0) > 0$ and set $f_n := (d_p(\mu_0))^{-\frac{1}{p}} (1-\phi_{n,p})$. Our first goal is to show that
\begin{equation} \label{(e3.4)}
    \lim_{n,j \to \infty} \left\| \frac{1}{2} \left( \frac{f_n}{\|f_n \|_0} + \frac{f_j}{\|f_j\|_0} \right) \right\|_0 = 1 \, .
\end{equation}
Using elementary properties of norms we obtain
\begin{equation} \label{(e3.5)}
\begin{aligned}
  1 & \geq \left\| \frac{1}{2} \left( \frac{f_n}{\|f_n \|_0} + \frac{f_j}{\|f_j\|_0} \right) \right\|_0 = \frac{1}{2} \left\| f_n -\left(f_n - \frac{f_n}{\| f_n \|_0} \right) +f_j - \left( f_j - \frac{f_j}{\| f_j \|_0} \right) \right\| \\ 
     & \geq  \frac{1}{2} \|f_n+f_j\|_0 - \frac{1}{2} \left\| f_n - \frac{f_n}{\| f_n \|_0} \right\|_0 - \frac{1}{2} \left\| f_j - \frac{f_j}{\| f_j \|_0} \right\|_0 \, . 
\end{aligned}
\end{equation}
From $\lim \| f_n \|_0 =(d_p(\mu_0))^{-\frac{1}{p}} \lim \|1-\phi_{n,p} \|_0 = (d_p(\mu_0))^{-\frac{1}{p}} \lim (d_p(\mu_n, \mu_0))^{\frac{1}{p}} = 1$ we conclude
\begin{equation} \label{(e3.6)}
   \lim \left\| f_n - \frac{f_n}{\| f_n \|_0} \right\|_0 =0 \, ,
\end{equation}
and hence,
\begin{equation}  \label{(e3.7)}
   \lim_{n,j \to \infty} \frac{1}{2} \| f_n+f_j \| \leq 1
\end{equation}
by (\ref{(e3.5)}). On the other hand $\frac{1}{2} (f_n+f_j) =(d_p(\mu_0))^{-\frac{1}{p}} (1-\frac{\phi_{n,p}+\phi_{j,p}}{2} )$. and since $\frac{\phi_{n,p}+\phi_{j,p}}{2} \in [{\mathcal T}(S)]_{\mu_0,p}$, we obtain $\| \frac{1}{2} (f_n+f_j) \|_0 = (d_p(\mu_0))^{-\frac{1}{p}} \|1-\frac{\phi_{n,p}+\phi_{j,p}}{2} ) \|_0 \geq (d_p(\mu_0))^{-\frac{1}{p}} (d_p(\mu_0))^{\frac{1}{p}} = 1$. Taking into account (\ref{(e3.7)}) we can conclude that
\begin{equation}  \label{(e3.8)}
    \lim_{n,j \to \infty}  \| \frac{1}{2} (f_n+f_j) \|_0 = 1
\end{equation}
and an application of (\ref{(e3.5)}), (\ref{(e3.6)}) and (\ref{(e3.8)}) yields (\ref{(e3.4)}), what means that for $\delta \in (0,\infty)$, there exists $n_0 \in \mathbb N$  such that $\left\| \frac{1}{2} \left( \frac{f_n}{\|f_n \|_0} + \frac{f_j}{\|f_j\|_0} \right) \right\|_0 > 1-\delta$ for all $n,j \geq n_0$. If it would exist a positive number $\varepsilon$ and a subsequence $(f_{n_r})_{r \in \mathbb N}$ such that 
\begin{equation} \label{(e3.9)}
   \left\| \left( \frac{f_{n_r}}{\|f_{n_r} \|_0} - \frac{f_{n_{r+1}}}{\|f_{n_{r+1}}\|_0} \right) \right\|_0 > \varepsilon \, , \,\, r \in \mathbb N \, ,
\end{equation}
then for any $\delta \in (0,\infty)$ there would exist elements $f_{n_r}$ and $f_{n_{r+1}}$ satisfying (\ref{(e3.9)}) and the inequality $\left\| \frac{1}{2} \left( \frac{f_{n_r}}{\|f_{n_r} \|_0} - \frac{f_{n_{r+1}}}{\|f_{n_{r+1}}\|_0} \right) \right\|_0 > 1-\delta$, a contradiction to the uniform rotundity of $\| \cdot \|_0$. Therefore, the sequence $\left( \frac{f_n}{\| f_n\|_0} \right)$ is a Cauchy sequence with respect to the norm $\| \cdot \|_0$. Since $\lim\| f_n \|_0 = 1$, the sequence $(\phi_{n,p})$ is a Cauchy sequence as well. Its limit $\lim\phi_{n,p} =: \phi_p$ satisfies the equality 
\begin{eqnarray*}
\| 1-\phi_p\|^p_0 & = & \lim \| 1-\phi_{n,p}\|^p_0 \\
                           & = & \lim d_p(\mu_n,\mu_0) \\
                           & = & d_p(\mu_0) \, ,
\end{eqnarray*}
what implies that $\phi_p=\phi_{0,p}$ since $\phi_p\in [{\mathcal T}]_{\mu_0,p}$ and the metric projection is unique.
\end{proof}

Note, that in case $p=2$ the assertion of the preceding  theorem can be easily derived from the Pythagorean theorem. Indeed, if $(1-\phi_{n,2})$ is the orthogonal sum of $(1-\phi_{0,2})$ and $(\phi_{0,2}-\phi_{n,2})$, one has $\int |1-\phi_{n,2}|^2 \, d\mu_0 = \int |1-\phi_{0,2}|^2 \, d\mu_0 + \int |\phi_{0,2}-\phi_{n,2}|^2 \, d\mu_0$, showing that (R2) implies (R3). 

Summarizing our extensions of Theorem \ref{t1.1} we state the following theorem. Its assertions fail if weak* convergence is replaced by convergence in measure, cf.~Example \ref{e5.2}.

\begin{theorem} \label{t3.10}
  Let $\rm w^*$-$\lim \mu_n =\mu_0$ and $\mu_n \geq c_n \mu_0$ for a sequence $(c_n)$ of positive real numbers with $\lim \inf c_n \geq 1$. Relation (R1) is satisfied for all sets $S$ and all $p \in (0,\infty)$. If, additionally, $\mu_0 << \mu_n$, $n \in \mathbb N$, and $\phi_{n,p}(S)$ exists, then (R2) is true. If $\phi_{0,p}(S)$ exists as well and $p \in (1,\infty)$, then (R3) is satisfied.
\end{theorem}

\begin{proof} 
The first assertion is Corollary \ref{c3.3}(ii). To prove the second assertion note first that $\lim \sup d_p(\mu_n,\mu_0) \leq \lim \sup \frac{1}{c_n} d_p(\mu_n) \leq d_p(\mu_0)$ by Theorem \ref{t3.1}. Since $\phi_{n,p} \in [{\mathcal T}]_{\mu_0,p}$, we have $\int |1-\phi_{n,p}|^p \, d\mu_0 \geq \int |1-\phi_{0,p}|^p \, d\mu_0$, and hence, \linebreak[4]
$\lim \inf d_p(\mu_n,\mu_0) \geq d_p(\mu_0)$. So (R2) follows and the third assertion is a consequence of Theorem \ref{t3.9}.
\end{proof}

Corollary \ref{c3.3} and the Examples \ref{e3.6}(b) and (c) show that there exist a set $S$, a measure $\mu_0$ and a sequence $(\mu_n)$ such that $\lim\|\mu_0-\mu_n\|=0$, $\lim d_2(\mu_n ; S)=d_2(\mu_0 ; S)$ and $\lim d_2(\mu_n,\mu_0 ; S)\not= d_2(\mu_0 ; S)$. Thus, norm convergence of $(\mu_n)$ is not sufficient for the implications  (R1) $\to$ (R2) and (R1) $\to$ (R3). Example \ref{e7.5}(a) and Theorem \ref{t7.6}(i)  will reveal that the condition $\rm m$-$\lim \mu_n =\mu_0$ is not sufficient for the implication (R3) $\to$ (R1), whereas $\lim \|\mu_0-\mu_n\|=0$ will be a sufficient condition, cf.~Corollary \ref{c3.12}.

\begin{lemma} \label{l3.11}
Let $\mu_0 << \mu_n$, $n \in \mathbb N$. Assume that the set $S$, $p \in (0,\infty)$ and $\mu_k$ are such that the metric projection $\phi_{k,p}(S)$ exists, $k \in {\mathbb N}_0$. Assume further that the following conditions are satisfied: 
\begin{itemize}
\item{(i)} $\rm m$-$\lim \mu_n = \mu_0$,
\item{(ii)} $\lim \sup d_p(\mu_n ; S) \leq d_p(\mu_0 ; S)$,
\item{(iii)} $\mu_0$-$\lim |1-\phi_{n,p}(S)| = |1-\phi_{0,p}(S)|$.
\end{itemize}
Then (R1) is true.
\end{lemma}

\begin{proof}
Let $\nu \in \mathcal M$ be such that $\mu_k << \nu$. Set $w_k:=\frac{d\mu_k}{d\nu}$, $k \in {\mathbb N}_0$. If (R1) would not be true, by (ii) there would exist a subsequence $(\mu_{n'})$ satisfying 
\begin{equation}   \label{(e3.10)}
    \lim_{n'\to \infty} d_p(\mu_{n'} ; S) < d_p(\mu_0 ; S) \, .
\end{equation}
From (iii) we derive the existence of another subsequence $(n'')$ of $(n')$ with
\begin{equation} \label{(e3.11)}
   \lim_{n'' \to \infty} | 1-\phi_{n'',p}(S) |^p =|1-\phi_{0,p}(S) |^p \quad \mu_0{\rm -a.e.}
\end{equation}
and (i) yields the existence of a subsequence $(n''')$ of $(n'')$ such that
\begin{equation} \label{(e3.12)}
   \lim_{n''' \to \infty} w_{n'''} = w_0 \quad \nu{\rm -a.e.} 
\end{equation}
Taking into account (\ref{(e3.11)}) and (\ref{(e3.12)}) and applying Fatou's lemma, we get 
\begin{eqnarray*}
   \lim\inf_{n''' \to \infty} d_p(\mu_{n'''} ; S) & =     & \lim \inf_{n''' \to \infty} \int |1-\phi_{n''',p}(S)|^p w_{n'''} \, d\nu \\
                                                                 & \geq & \int |1-\phi_{0,p}(S)|^p w_{0} \, d\nu = d_p(\mu_0) \, ,
\end{eqnarray*}
a contradiction to (\ref{(e3.10)}).
\end{proof}

\begin{corollary} \label{c3.12}
   If $\lim \| \mu_0-\mu_n \|=0$, then (R3) implies (R1).
\end{corollary}

\begin{lemma} \label{l3.13}
   Let $S$, $p$ and $\mu_k$ be such that $\mu_0 << \mu_k$, $\phi_{k,p}(S)$ exists, the sequence $({\esssup}_{\mu_0} | \phi_{n,p}(S) |)$ is bounded and $\mu_0$-$\lim \phi_{n,p}(S) = \phi_{0,p}(S)$. Then (R3) is true. If, additionally, $\lim \|\mu_0-\mu_n \|=0$ then (R1) is also true.
\end{lemma}

\begin{proof}
The suppositions of the lemma imply that the sequence ${\esssup}_{\mu_0} | \phi_{0,p}(S) - \phi_{n,p}(S) |^p)$ is a bounded sequence and (R3) follows from Lebesgue's dominated convergence theorem. The second assertion then follows from Corollary \ref{c3.12}.
\end{proof}

The strong boundedness condition on the sequence $(\phi_{n,p}(S))$ is the main obstacle to apply the preceding lemma. For the by far most important case  $p=2$, now we state a result which somewhat widens its field of applications. Let $S$ be an arbitrary non-empty subset of $G \setminus \{ 0 \}$ and $x \in G \setminus \{ 0 \}$ be such that $e_x \not\in [{\mathcal T}(S)]_{\mu_0,2}$. Denote by $P_k$ and $Q_k$ the orthoprojections in $L^2(\mu_k)$ onto $[{\mathcal T}(S)]_{\mu_k,2}$  and $[{\mathcal T}(S \cup \{x\})]_{\mu_k,2}$, resp., $k \in {\mathbb N}_0$. 

\begin{proposition} \label{p3.14}
   Let $\mu_0 << \mu_n$, $n \in {\mathbb N}_0$, and $\lim\|\mu_0 - \mu_n \|=0$. Assume that for all $y \in G$, the function $P_k e_y$ has the following two properties:
\begin{itemize}
\item{(i)} The sequence $({\esssup}_{\mu_n} |P_n e_y |)$ is bounded,
\item{(ii)} $\lim \int |P_0 e_y - P_n e_y |^2 \, d\mu_0=0$ .
\end{itemize}
Then for all $y \in G$, the function $Q_k e_y$ has the same properties.
\end{proposition}

To prove Proposition \ref{p3.14} we show the following lemma:

\begin{lemma} \label{l3.15}
   Let $\mu_0 << \mu_n$ and $\lim \| \mu_0 - \mu_n \|=0$. If $(f_n )$ denotes a sequence of complex-valued measurable functions such that 
\begin{equation}  \label{(e3.13)}
   \sup \{ {\esssup}_{\mu_n} |f_n| : n \in {\mathbb N} \} =: c < \infty
\end{equation}
and $\rm \mu_0$-$\lim f_n =: f_0$ exists, then $\lim \int f_n \, d\mu_n = \int f_0 \, d\mu_0$.
\end{lemma}

\begin{proof}
We have
\begin{equation} \label{(e3.14)}
   \left| \int f_0 \, d\mu_0 - \int f_n \, d\mu_n \right| \leq \int | f_0-f_n| \, d\mu_0+ \int | f_n| \, d|\mu_0-\mu_n| \,\, ,
\end{equation}
where $|\mu_0-\mu_n|$ denotes the variation of the real-valued measure $\mu_0-\mu_n$. Since $\mu_0 << \mu_n$ condition (\ref{(e3.13)}) implies that $\sup \{ {\esssup}_{\mu_0} | f_n | : n \in {\mathbb N}\}\leq c$. Thus, if $n \to \infty$ the first summand at the right-hand side of (\ref{(e3.14)}) tends to zero by Lebesgue's dominated convergence theorem. To see that the second  summand tends to zero as well, note first that $|\mu_0-\mu_n| << \mu_n$. In fact, if $B \in\mathcal B$ with $\mu_n(B)=0$ then $\mu_n(C)=\mu_0(C)=0$ for any $C \in B$, $C \subseteq B$, what yields $|\mu_0-\mu_n|(B)=0$. Therefore, $\sup \{ {\esssup}_{|\mu_0-\mu_n|} | f_n | : n \in {\mathbb N} \} \leq c$ by (\ref{(e3.13)}), and hence, $\lim \int | f_n | \, d|\mu_0-\mu_n| \leq c \lim \| \mu_0-\mu_n \| =0$.
\end{proof}

{\it Proof of Proposition \ref{p3.14}.} Denote by $( \cdot , \cdot )_k$ and $\| \cdot \|_k$ the scalar product and the norm, resp., of $L^2(\mu_k)$. For $y \in G$, under the assumption $\| e_x -P_k e_x \| \not= 0$ we can write
\begin{equation} \label{(e3.15)}
   Q_k e_y := P_k e_y + \left( e_y , \frac{e_x-P_k e_x}{\| e_x - P_k e_x \|_k} \right)_k  \frac{e_x-P_k e_x}{\| e_x - P_k e_x \|_k} \, . 
\end{equation}
The claim of the proposition can be derived from (\ref{(e3.15)}) applying Lemma \ref{l3.15} twice. First, setting $\phi_k:= | e_x-P_k e_x |^2$ in Lemma \ref{l3.15} , we get $\lim \| e_x- P_n e_x \|_n = \| e_x -P_0 e_x \|_0$, in particular, $\| e_x- P_n e_x \|_n \not=0$ for any sufficiently large $n$ since the assumption $e_x \not\in [{\mathcal T}(S)]_{\mu_0,2}$ yields $\| e_x -P_0 e_x \|_0 > 0$. Now, setting $f_k := e_y \frac{\overline{(e_x -P_k e_x)}}{\| e_x - P_k e_x \|_k}$ in Lemma \ref{l3.15}, we obtain $\lim \left( e_y , \frac{e_x -P_n e_x}{\|e_x - P_ne_x \|_n}\right)_n =  \left( e_y , \frac{e_x -P_0 e_x}{\|e_x - P_0e_x \|_0}\right)_0$, and consequently, $\sup \{ {\esssup}_{\mu_n} |Q_n e_y | : n \in {\mathbb N} \}< \infty$ by condition (i) and (\ref{(e3.15)}). Finally, $\lim \int |Q_0 e_y -Q_n e_y |^2 \, d\mu_0 = 0$ follows from condition (ii) and (\ref{(e3.15)}).
$\Box$  \linebreak[4]

We conclude the present section with a few remarks on the estimated error of the theoretical prediction $d_p(\mu_o , \mu_n):= \int |1-\phi_{0,p}|^p \, d\mu_n$, of course, presuming that $\phi_{0,p}$ exists and $\mu_n << \mu_0$, $n \in {\mathbb N}$. From the mere definition of weak* convergence we can derive that, in case $\phi_{0,p}$ is continuous and bounded, then $\lim d_p(\mu_0,\mu_n)=d_p(\mu_0)$ for every sequence $(\mu_n)$ satisfying $\rm w^*$-$\lim \mu_n = \mu_0$. Similarly, if $\phi_{0,p} \in L^\infty (\mu_0)$, then $\lim d_p(\mu_0, \mu_n)=d_p(\mu_0)$ for any sequence $(\mu_n)$ with $\rm w$-$\lim \mu_n = \mu_0$. Conversely, it $\phi_{0,p} \not\in L^\infty(\mu_0)$, there exists a sequence $(\mu_n)$ such that $\lim \| \mu_0-\mu_n \|=0$ and  $\lim d_p(\mu_0,\mu_n)=\infty$. To give an example, set $B_n := \{ \gamma \in \Gamma : | \phi_{0,p}(\gamma)| \geq n^2 \}$, $b_n:= \mu_0(B_n) >0$, and note that $\lim b_n = 0$. Defining $d_{\mu_n} := ((n^p b_n)^{-1} \cdot 1_{B_n}+ 1_{\Gamma \setminus B_n}) \, d\mu_0$, we obtain $\lim\| \mu_0 - \mu_n \| \leq \lim (1+(n^pb_n)^{-1})b_n=0$ and $\lim \int | 1-\phi_{0,p}|^p \, d\mu_n \geq \lim (n^2-1)^p (n^pb_n)^{-1} b_n =\infty$. It is also worth mentioning that if the observation at the points of $S$ give full information on the process, i.e. if $d_p(\mu_0 ; S)=0$, or equivalently, $\phi_{0,p}(S)=1$, then $\phi_{0,p}(S)=1$ $\mu$-a.e., since $\mu_n << \mu_0$. Thus, if $d_p(\mu_0;S)=0$, then $d_p(\mu_0,\mu_n ;S)=0$ always. Conversely, if $d_p(\mu_0 ;S) >0$, under a slight technical proviso on $\mu_0$ one can construct a sequence $(\mu_n)$ satisfying $\rm m$-$\lim \mu_n=\mu_0$ and $\lim d_p(\mu_0,\mu_n;S) > d_p(\mu_0;S)$. To do this note first that there exist $B_n \in  {\mathcal B}$ and $c \in (0,\infty)$ such that $\mu_0(B) >0$ and $|1-\phi_{0,p}(S)|^p \geq c$ $\mu_0$-a.e. on $B$. If we assume now that there exists a sequence $(B_n)$ of Borel subsets of $B$ with $\mu_0(B_n) >0$ and $\lim\mu_0(B_n)=0$, we can set $d\mu_n := \left( \frac{1}{c \mu_0(B_n)} \cdot 1_{B_n} + 1_{\Gamma \setminus B_n} \right) \, d\mu_0$ and compute $\lim \int | 1-\phi_{0,p}(S)|^p \, d\mu_n \geq \lim (1+\int_{\Gamma \setminus B_n} |1-\phi_{0,p}(S) |^p \, d\mu_n) =1+d_p(\mu_0;S)$.

%============================================%

\section{Interpolation of one missing value}

Throughout this section, $S$ denotes the subset $S=G\setminus \{ 0 \}$ of an LCA group $G$, and the dependence on $S$ will not be indicated in most cases. 

\begin{lemma} \label{l4.1}
   If $G$ is not discrete, then $[{\mathcal T}]_{\mu,p} = L^p(\mu)$ for all $\mu \in \mathcal M$ and $p \in (0,\infty )$.
\end{lemma}

\begin{proof}
Since $G$ is the dual group of its dual group $\Gamma$, the family of subsets $U(K,\delta):= \{ x \in G : | e_0(\gamma)-e_x(\gamma)|< \delta \,\, {\rm for} \,\, {\rm all} \,\, \gamma \in K\}$, where $K$ runs through the compact subsets of $\Gamma$ and $\delta$ through the positive real numbers constitutes a basis of neighbourhoods of the zero element $0$ of $G$. Let $\varepsilon \in (0,\infty )$. By the regularity of the measure $\mu$, there exists a compact subset $K \in \Gamma$ satisfying $\mu(\Gamma \setminus K) < \frac{\varepsilon}{2^{p+1}}$. Choose $\delta < \left( \frac{\varepsilon}{2 \mu(K)} \right)^\frac{1}{p}$ and obtain $\int |1-e_x| \, d\mu < \varepsilon$ for all $x \in U(K,\delta )$. Since $G$ is assumed not to be discrete, there exists an $x \in U(K,\delta)$ which is different from $0$. The assertion follows.
\end{proof}

 According to the preceding lemma we can and shall assume in the present section that the group $G$ is discrete. Thus, $\Gamma$ is compact and $\lambda$ denotes its normalized Haar measure. Moreover, since $d_p(\mu)=0$ for all $\mu \in \mathcal M$ and $p\in (0,1)$, cf.~\cite[Thm.~4.6]{[12]}, it is sufficient to study the case $p \in [1,\infty)$. As a third simplification we are concerned with absolutely continuous measures only: $d\mu = w \, d\lambda$. For if $d\mu = w \, d\lambda + \mu_\sigma$, where $\mu_\sigma$ is the singular part of $\mu$, then
\begin{equation}  \label{(e4.1)}
   d_p(\mu) = d_p(w \, d\lambda) \, , \,\, p \in (0,\infty) \, .
\end{equation}
If $p \in (1,\infty)$, relation (\ref{(e4.1)}) is a consequence of Theorem 4.1 and Lemma 4.3 of \cite{[13]}. Note, that an analogous result to that of Theorem 4.1 of \cite{[13]}, which was stated there for $p\in (1,\infty)$, is true for $p\in (0,1]$ since Lemma 2.3 of \cite{[13]} which is needed for its proof as well as the condition $L^p(\mu_\sigma) \subseteq [{\mathcal T}]_{\mu,p}$ of \cite[(4.1)]{[13]} remain true for $p \in (0,1]$. 

Let $p\in (1,\infty)$ and $q$ its conjugate number, i.e. $\frac{1}{p}+\frac{1}{q}=1$. The following results were established by Weron \cite{[31]}. If $w^{-\frac{q}{p}} \in L^1(\lambda)$, the metric projection $\phi_{\mu,p}$ is equal to
\begin{equation}  \label{(e4.2)}
    \phi_{\mu,p}= 1- \left( \int w^{-\frac{q}{p}} \, d\lambda \right)^{-1} w^{-\frac{q}{p}}
\end{equation}
which yields
\begin{equation}   \label{(e4.3)}
   d_p(\mu) = \left( \int w^{-\frac{q}{p}} \, d\mu \right)^{-\frac{p}{q}} \, .
\end{equation}
If $w^{-\frac{q}{p}} \not\in L^1(\mu)$, in particular, if $w=0$ on a set of positive $\lambda$-measure, then $\phi_{\mu,p}= 1$ and $d_p(\mu)=0$. As was mentioned in \cite{[19]}, one has
\begin{equation}    \label{(e4.4)}
   d_1(\mu)= {\essinf}_\lambda w \, .
\end{equation}
Therefore, the metric projection $\phi_{\mu,1}$ is equal to $1$ if and only if $\essinf_\lambda w = 0$. Let $B:= \{\gamma \in \Gamma :  w(\gamma)=\essinf_\lambda w \}$.

\begin{proposition} \label{p4.2}
   Let $\essinf_\lambda w > 0$. If $\lambda(B)>0$, there exist infinitely many functions $\psi \in [{\mathcal T}]_{\mu,1}$ satisfying
\begin{equation}  \label{(e4.5)}
   \int |1-\psi | \, d\mu = d_1(\mu) \, .
\end{equation}
In case $\lambda(B)=0$ such a function does not exist.
\end{proposition}

To prove Proposition \ref{p4.2} we need a characterization of the set $[{\mathcal T}]_{\mu,1}$ in $L^1(\mu)$. The following description is more general and of interest in its own.

\begin{lemma}   \label{l4.3}
   Let $p \in [1,\infty)$. If $[{\mathcal T}]_{\mu,p} \not= L^p(\mu)$, then all functions of $L^p(\mu)$ are integrable with respect to $\lambda$ and $[{\mathcal T}]_{\mu,p}$ is exactly the subspace of all $f \in L^1(\mu)$ with $\int f \, d\lambda =0$. 
\end{lemma}

\begin{proof}
Let $p\in (1,\infty)$. If $w^{-\frac{q}{p}} \in L^1(\mu)$, then for $f \in L^p(\mu)$ H\"olders inequality yields
\begin{equation}  \label{(e4.6)}
   \int |f| \, d\lambda = \int |f|  w^\frac{1}{p} w^{-\frac{1}{p}} \, d\lambda \leq \left( \int |f|^p w \, d\lambda \right)^\frac{1}{p} \left( \int w^{-\frac{q}{p}} \, d\lambda \right)^\frac{1}{q} < \infty \, .
\end{equation}
Moreover, (\ref{(e4.6)}) reveals that $\int f \, d\lambda =0$ for all $f \in [{\mathcal T}]_{\mu,p}$. Conversely, if $f \in L^p(\mu)$ and $\int f \, d\lambda =0$, then  $f \in [{\mathcal T}]_{\mu,p}$ since  $[{\mathcal T}]_{\mu,p}$ has codimension $1$ and the non-zero constants do not belong to  $[{\mathcal T}]_{\mu,p}$. If $p=1$ the proof is similar and, thus, is omitted.
\end{proof}

{\it Proof of Proposition \ref{p4.2}.} Let $c:= \essinf_\lambda w >0$. By Lemma \ref{l4.3}, a function $\psi \in L^1(\mu)$ is an element of  $[{\mathcal T}]_{\mu,1}$ and satisfies (\ref{(e4.5)}) if and only if 
\begin{equation}   \label{(e4.7)}
    \int \psi \, d\lambda = 0
\end{equation}
and
\begin{equation}    \label{(e4.8)}
   \int |1-\psi| w \, d\lambda =c \, .
\end{equation}
From (\ref{(e4.7)}) one concludes
\begin{equation}     \label{(e4.9)}
   \int (1-\psi) \, d\lambda = 1
\end{equation}
and (\ref{(e4.8)}) implies that $c=\int |1-\psi| w \, d\lambda \geq c \cdot \int |1-\psi| \, d\lambda$, hence $\int |1-\psi| \, d\lambda \leq 1$. Taking into account (\ref{(e4.9)}), we obtain $1-\psi \geq 0$ and (\ref{(e4.8)}) transforms to
\begin{equation}   \label{(e4.10)}
    \int (1-\psi)w \, d\lambda = c \, .
\end{equation}
 If $\lambda(B) >0$ one easily concludes that any function $\psi \in L^1(\mu)$ with $\psi=1$ on $\Gamma \setminus B$,  $1-\psi \geq 0$, and $\int_B \psi \, d\lambda = - \lambda(\Gamma \setminus B)$ satisfies (\ref{(e4.7)}) and (\ref{(e4.8)}). Let $\lambda(B)=0$. If a function $\psi \in L^1(\mu)$ would satisfy (\ref{(e4.9)}), (\ref{(e4.10)}) and the inequality $1-\psi \geq 0$, it would follow $1= \int (1-\psi) \, d\lambda \leq \esssup_\lambda \frac{1}{w} \cdot \int (1-\psi)w \, d\lambda = \frac{1}{c} \cdot c = 1$. So $w$ would be a constant, say $C$, on the set $A:=\{ \gamma \in \Gamma : \psi(\gamma) < 1 \}$. Since $\lambda(B) =0$ we have $C >c$ and $\lambda(A) >0$. Consequently, $c=\int_A (1-\psi)w \, d\lambda = C\cdot \int (1-\psi) \, d\lambda  = C$, a contradiction.
$\Box$ \linebreak[4]

Recall, that a subspace $L$ of $L^p(\mu)$ is a Cheby{\v{s}}ev subspace if for arbitrary $f \in L^p(\mu)$ the metric projection onto $L$ exists. Cheby{\v{s}}ev subspaces of $L^1(\mu)$ with finite codimension were described by Garkavi \cite[Thm.~3]{[8]}. From our results we can derive the following corollary: 

\begin{corollary}   \label{c4.4}
   Let $p \in (0,\infty)$ and $\mu \in \mathcal M$. If $p \not= 1$, the space $[{\mathcal T}(G \setminus \{0 \})]_{\mu,p}$ is a Cheby{\v{s}}ev subspace of $L^p(\mu)$. The space $[{\mathcal T}(G \setminus \{0 \})]_{\mu,1}$ is a Cheby{\v{s}}ev subspace of $L^1(\mu)$ if and only if $\essinf_\lambda w =0$, where $w$ denotes the Radon-Nikodym derivative of the absolutely continuous part of $\mu$.
\end{corollary}

In the remaining part of this section some of the general results  of Section 3 are specified to the case $S=G \setminus \{ 0 \}$. Let us assume that $\mu_k\in \mathcal M$ is absolutely continuous and set $w_k:=\frac{d\mu_k}{d\lambda}$, $ k \in {\mathbb N}_0$.

\begin{theorem}    \label{t4.5}
   Let $p \in (1,\infty)$. The relations (R2) and (R3) are equivalent for $S=G \setminus \{ 0 \}$.
\end{theorem}

\begin{proof} According to Theorem \ref{t3.9} we have only to prove that in case $d_p(\mu_0) \not= 0$ and $\lim d_p(\mu_n,\mu_0)=d_p(\mu_0)$, $\phi_{n,p}\in [{\mathcal T}]_{\mu_0,p}$ for all $n$ large enough. Since the just mentioned conditions imply that $[{\mathcal T}]_{\mu_n,p} \not=L^p(\mu_n)$ and $\phi_{n,p} \in L^p(\mu_0)$ for all $n$ large enough, we get $\int \phi_{n,p}\, d\lambda = 0$, hence $\phi_{n,p}\in [{\mathcal T}]_{\mu_0,p}$ by Lemma \ref{l4.3}. 
\end{proof}

The next theorem shows that for $S=G \setminus \{ 0 \}$ weak*-convergence can be replaced by convergence in measure in Theorem \ref{t3.1}.

\begin{theorem}  \label{t4.6}
   If $\rm m$-$\lim \mu_n = \mu_0$, then inequality (\ref{(e3.3)}) is satisfied for $S=G \setminus \{ 0 \}$ and all $p \in [1,\infty)$.
\end{theorem}

\begin{proof}
If $p \in (1,\infty)$, the assertion is an easy consequence of (\ref{(e4.3)}) and Fatou's lemma. Let $p=1$. Assume that there exists a subsequence $(w_{n'})$ such that $\lim_{n' \to \infty} \essinf_\lambda w_{n'} =:a > d_1(\mu_0)$. Then $w_{n'} > d_1(\mu_0) + \frac{3}{4} (a-d_1(\mu_0))$ $\lambda$-a.e. for all sufficiently large $n'$. On the other hand, there exists a set $B \in \mathcal B$ with $\lambda(B) > 0$ and $w_0 < d_1(\mu_0)+ \frac{1}{4} (a-d_1(\mu_0))$ on $B$, what implies that for all sufficiently large $n'$ we have $|w_0-w_{n'}| > \frac{1}{2} (a-d_1(\mu_0))$ $\lambda$-a.e. on $B$. This is a contradiction to m-$\lim \mu_n = \mu_0$.
\end{proof}

\begin{corollary}  \label{c4.7}
   If $\rm m$-$\lim \mu_n = \mu_0$ and $d_p(\mu_0)=0$, then (R1) is satisfied for $S=G \setminus \{ 0 \}$ and all $p \in [1,\infty)$.
\end{corollary}

\begin{corollary}   \label{c4.8}
   Let $p \in (1, \infty)$. If $\rm m$-$\lim \mu_n =\mu_0$, then for $G=S \setminus \{ 0 \}$ the relation (R3) yields (R1).
\end{corollary}

\begin{proof}
Apply Lemma \ref{l3.11} and Theorem \ref{t4.6}.
\end{proof}

Let $\mu, \nu \in \mathcal M$ be absolutely continuous and $\mu << \nu$. Let $w:=\frac{d\mu}{d\lambda}$, $v:= \frac{d\nu}{d\lambda}$. If $v^{-\frac{q}{p}} \not\in L^1(\lambda)$ then $d_p(\nu,\mu)=0$. If $v^{-\frac{q}{p}} \in L^1(\lambda)$ then (\ref{(e4.2)}) yields 
\begin{equation}   \label{(e4.11)}
   d_p(\nu,\mu) =\left( \int v^{-\frac{q}{p}} \, d\lambda \right)^{-p} \cdot \int v^{-q}w \, d\lambda = (d_p(\nu))^q \cdot  \int v^{-q}w \, d\lambda \,\,\,\, {\rm for}\,\,\, p \in (1,\infty) \, .
\end{equation}

\begin{theorem}   \label{t4.9}
   If $p\in (1,\infty)$ then $d_p(\nu,\mu)=0$ or it satisfies the inequality 
\begin{equation}   \label{(e4.12)}
   d_p(\mu) \leq d_p(\nu , \mu) \leq \infty \, .
\end{equation}
If $d_p(\mu) >0$, the left inequality of (\ref{(e4.12)}) becomes an equality if and only if $\nu = a\mu$ for some number $a \in (0,\infty)$.
\end{theorem}

Note, that the cases $d_p(\mu) >0$ and $d_p(\nu,\mu)=0$ as well as $d_p(\nu,\mu)=\infty$ can occur, cf.~Examples \ref{e3.2} and \ref{e3.6}. If both $d_p(\mu)$ and $d_p(\nu,\mu)$ are positive real numbers, the left inequality of (\ref{(e4.12)}), which was obtained for $p=2$ by Taniguchi \cite[p.~57]{[30]} in a slightly more general form, is a simple consequence of H\"older's inequality. To see this, multiply the left inequality of (\ref{(e4.12)}) by $(\int w^{-\frac{q}{p}} \, d\lambda)^\frac{p}{q} (\int v^{-\frac{q}{p}}w \, d\lambda)^p$ and take into account (\ref{(e4.3)}) and (\ref{(e4.11)}). We get the inequality $(\int v^{-\frac{q}{p}} \, d\lambda)^p \leq \int v^{-q} w \, d\lambda \cdot (\int w^{-\frac{q}{p}} \, d\lambda)^\frac{p}{q}$, which can be proved applying H\"older's inequality to the integral $(\int (v^{-\frac{q}{p}} w^\frac{1}{p}) w^{-\frac{1}{p}} \, d\lambda)^p$. Here we give a different proof of Theorem \ref{t4.9} applying Lemma \ref{l4.3}. 

{\it Proof of Theorem \ref{t4.9}.} If $d_p(\mu)=0$ or $d_p(\nu,\mu)=\infty$ there is nothing to prove. Let both $d_p(\mu)$ and $d_p(\nu,\mu)$ be positive numbers. Since Lemma \ref{l4.3} implies that $\phi_{\nu,p} \in [{\mathcal T}]_{\mu,p}$ we get $d_p(\mu) = \int |1-\phi_{\mu,p}|^p \, d\mu \leq \int |1-\phi_{\nu,p}|^p \, d\mu = d_p(\nu,\mu)$ by the definition of the metric projection. Taking into account the remarks after Theorem \ref{t4.9}, we can conclude that the left inequality of (\ref{(e4.12)}) becomes an equality if and only if H\"older's inequality $\int v^{-\frac{q}{p}} w^\frac{1}{p} w^{-\frac{1}{p}} \, d\lambda \leq (\int v^{-q} w \, d\lambda)^\frac{1}{p} (\int w^{-\frac{q}{p}}\, d\lambda)^\frac{1}{q}$ becomes an equality, i.e. if and only if $(v^{-\frac{q}{p}} w^\frac{1}{p})^p = a^{-q} w^{-\frac{q}{p}}$ or, equivalently, $v=aw$ for some $a \in (0,\infty)$.
$\Box$ \linebreak[4]

The assertion of Theorem \ref{t4.9} should be compared with Example \ref{e3.6}, which exhibits that for any $c \in [0,\infty)$ and any $a \in \{0 \} \cup [c,\infty)$ there exist $\mu_0 \in \mathcal M$ and a sequence $(\mu_n)$ of measures of $\mathcal M$ such that $\lim \| \mu_0-\mu_n\|=0$, $d_2(\mu_0)=c$ and $\lim d_2(\mu_n,\mu_0)=a$. 

\begin{corollary}  \label{c4.10}
   Let $p \in (1,\infty)$. If $\rm m$-$\lim \mu_n = \mu_0$ and $\mu_n \geq c\mu_0$, $n \in \mathbb N$, for some positive number $c$, then (R1) is satisfied for $S=G \setminus \{ 0 \}$. If, additionally, $\mu_0 << \mu_n$, then (R3) is satisfied as well.
\end{corollary}

\begin{proof}
To prove (R1), by Corollary \ref{c4.7} we can assume that $d_p(\mu_0)>0$ or, equivalently, $w_0^{-\frac{q}{p}} \in L^1(\lambda)$. Now, apply Lebesgue's dominated convergence theorem. To prove (R3), it is enough to prove (R2), cf.~Theorem \ref{t4.5}. Since $d_p(\mu_n,\mu_0) \leq \frac{1}{c} d_p(\mu_n)$, the relation (R2) follows in case $d_p(\mu_0)=0$ from the first part of the corollary. Let $d_p(\mu_0) >0$ or, what is the same, $w^{-\frac{q}{p}}_0 \in L^1(\lambda)$. Then $\lim \int w^{-\frac{q}{p}}_n\, d\lambda =\int w^{-\frac{q}{p}}_0 \, d\lambda$ by Lebesgue's dominated convergence theorem. Thus, from (\ref{(e4.11)}) we get $\lim d_p(\mu_n,\mu_0)=( \int w^{-\frac{q}{p}}_0 \, d\lambda)^{-p} \cdot \int w_0^{-\frac{q}{p}} \, d\lambda = (\int w_0^{-\frac{q}{p}} \, d\lambda)^{-\frac{p}{q}} = d_p(\mu_0)$
\end{proof}

A simple example $w_0=1$, $w_n= \frac{1}{2} \cdot 1_{[0,\frac{1}{n})} + 1_{[\frac{1}{n},2\pi)}$ on $\Gamma=[0,2\pi]$, cf.~Example \ref{e3.2}, reveals that the first assertion of the preceding corollary is false for $p=1$. Moreover, the following example shows that the convergence in measure cannot be replaced with weak convergence. 

\begin{example}   \label{e4.11}
Let $\Gamma = [0,2\pi)$ as in Example \ref{e3.5}. Choose $w_0=2$, $w_n(\gamma)= 2 + \cos(n\gamma)$, $\gamma \in[0,2\pi)$, $n \in \mathbb N$. Then $w_n \geq \frac{1}{2} w_0$, w-$\lim w_n = w_0$ by a well-known property of Fourier coefficients and simple computations yield $d_2(\mu_0)=2$, $d_2(\mu_n)= \sqrt{3}$ and $d_2(\mu_n,\mu_0)=\frac{4}{\sqrt{3}}$.
\end{example}

%%%%%%%%%%%%%%%%%%%%%%%%%%%%%%%%%%%%%

\section{Prediction $m$ steps ahead}

The goal of the present section is $m$-steps-ahead prediction, which has played a central role from the very beginning of the prediction theory. To avoid several transformations to complex conjugates it is more convenient to study the $m$-steps backwards prediction problem instead. So, let $G=\mathbb Z$, $\Gamma=[0,2\pi)$, $\lambda$ be the  normalized Lebesgue measure on $\Gamma$, and the character $e_x$ for $x \in \mathbb Z$ be defined as described  in Section 3. For $m \in \mathbb N$, denote by $S_m$ the set $S_m:= (m-1) +{\mathbb N} =  \{ m,m+1, ... \}$. Let $\mu \in \mathcal M$ and $d_\mu=w \, d\lambda +\mu_\sigma$ be its decomposition into its absolutely continuous and singular parts. Similarly to (\ref{(e4.1)}) the equality 
\begin{equation}   \label{(e5.1)}
   d_p(\mu;S_m) := d_p(w \, d\lambda; S_m)
\end{equation}
is satisfied for all $p \in (0,\infty)$ and $m\in \mathbb N$. For $p=2$ and $m=1$, the relation (\ref{(e5.1)}) is a well-known result, cf.~\cite{[10]}. An extension to $p \in [1,\infty)$ and $m \in \mathbb N$ was given by Koosis \cite[Section VII.C]{[18]} with the aid of some technical facts from function theory. A different proof and at the same time an extension to all $p \in (0,\infty)$ can be derived from \cite{[13]}, see the remarks after formula (\ref{(e4.1)}) of the present paper. Similarly to Section 4, we assume $\mu_\sigma=0$ throughout the present section.

The celebrated Szeg\"o theorem asserts that for all $p\in (0,\infty)$, $d_p(\mu;S_1)$ is equal to the geometric mean of $w$, i.e.
\begin{equation}   \label{(e5.2)}
   d_p(\mu;S_1) = \exp \int \log w \, d\lambda \, ,
\end{equation}
where the right hand side of (\ref{(e5.1)}) has to be understood as $0$ if $\log w \not\in L^1(\lambda)$, in particular, if $w=0$ on a set of positive $\lambda$-measure, cf.~\cite[p.~257]{[1]}. Setting $\log^+ w := \max\{\log w, 0 \}$ and $\log^- w := \max \{  -\log w, 0 \}$, we mention that $\log w= \log^+ w -\log^- w$ and that $\log w \in L^1(\lambda)$ if and only if $\log^- w \in L^1(\lambda)$ since $\log^+ w \leq w$ is integrable w.r.t. $\lambda$. Let $d_{\mu_k}:= w_k \, d\lambda$, $k \in {\mathbb N}_0$. Using (\ref{(e5.2)}) we can construct two counterexamples concerning the set $S_1$. 
Example \ref{e5.1} shows that even for a uniformly convergent sequence $(w_n)$ equality (R1) can fail. Example \ref{e5.2} reveals that analogues of Theorem \ref{t4.6}, Corollary \ref{c4.7} and Corollary \ref{c4.10} are not true for one-step-ahead prediction, and that the assertions of Theorem \ref{t3.1} and Corollary \ref{c3.3} can fail if weak* convergence of the sequence $(\mu_n)$ is replaced by its convergence in measure. Note, however, that if m-$\lim \mu_n=\mu_0$ and $\lim \int \log^+ w_n \, d\lambda=\int \log^+ w_0 \, d\lambda$, then inequality (\ref{(e3.3)}) issatisfied for $S=S_1$ since $\lim \int \log w_n\, d\lambda = \lim \int (\log^+ w_n - \log^- w_n)\, d\lambda \leq \int (\log^+ w_0 - \log^- w_0)\, d\lambda = \int \log w_0 \, d\lambda$ by Fatou's lemma, cf.~\cite[Thm.~2.3.3]{[20]}.

\begin{example}  \label{e5.1}
Let $w_0(\gamma)= \exp \frac{1}{\gamma \log^3 \gamma}$, $w_n(\gamma)= 1_{[0,\frac{1}{n})}(\gamma) \cdot \exp \frac{1}{\gamma \log \gamma} + 1_{[\frac{1}{n},2\pi)}(\gamma) w_0(\gamma)$, $\gamma \in [0,2\pi)$, $n \in \mathbb N$. From (\ref{(e5.2)}) we have $d_p(\mu_0;S_1) >0$ and $d_p(\mu_n;S_1) =0$. To show that $\lim w_n = w_0$ uniformly on $[0,2\pi)$ note first, that $0 \leq w_0-w_n \leq w_0$. Since $\frac{d(\gamma \log^3 \gamma)}{d \gamma}= (3+\log \gamma)\log^2\gamma$, the function $w_0$ is increasing on $(0,{\rm e}^{-3})$. Therefore, if $n > {\rm e}^3$ and $\gamma \in [0,\frac{1}{n})$, then  $| w_0(\gamma)-w_n(\gamma)| < \exp \left( - \frac{n}{\log^3 n}\right)$, which tends to zero for $n\to\infty$.
\end{example}

\begin{example}  \label{e5.2}
Set $w_0(\gamma):= \exp(-\frac{1}{\gamma})$, $\gamma \in [0,2\pi)$, $w_n:=\exp (n^2)\cdot 1_{[0,\frac{1}{n})} + w_0 \cdot 1_{[\frac{1}{n},2\pi)}$. Obviously, $w_n \geq w_0$, m-$\lim \mu_n =\mu_0$, and (\ref{(e5.2)}) yields $d_p(\mu_0;S_1)=0$ and $d_p(\mu_n;S_1)=\exp (n-\log n - 2 \log 2\pi)$, which tends to $\infty$ if $n \to\infty$. 
\end{example}

It is a simple consequence of (\ref{(e5.2)}) that the distance $d_p(\mu; S_m)$ is equal to zero if and only if $\log w \not\in L^1(\lambda)$, $p \in (0,\infty)$, $m\in \mathbb N$. The principal tool to study the case $d_p(\mu ; S_m) >0$ is Hardy space theory, cf.~the excellent books\cite{[6]}, \cite{[10]}, \cite{[18]}. If $\log w \in L^!(\lambda)$, the function $f$ defined by $f(z):= \frac{1}{2} \int \frac{e_1+z}{e_1-z} \log w \, d\lambda$ has a Taylor series
\begin{equation}   \label{(e5.3)}
     f(z) = \frac{a_0}{2} +\sum^\infty_{j=1} a_j z^j \,\, , |z|<1 \, ,
\end{equation}
where $a_j$ is the $j$-th Fourier coefficient of $\log w$, i.e. $a_j=\int e_{-j} \log w\, d\lambda$. The function
\begin{equation}   \label{(e5.4)}
    h:= \exp f
\end{equation}
plays a key role in prediction theory. It is an outer function of Hardy class $H^2$, $h(0)>0$, and the boundary function of $h$ on the unit circle, which is also denoted by $h$, satisfies the relation
\begin{equation}    \label{(e5.5)}
   | h ({\rm e}^{i\gamma}) |^2 =w(\gamma)
\end{equation}
for $\lambda$-almost any $\gamma \in [0,2\pi)$. Setting
\begin{equation}    \label{(e5.6)}
   b_j := \frac{h^{(j)}(0)}{j!} \,\, , j \in {\mathbb N}_0 \, ,
\end{equation}
the function $h$ has a Taylor expansion $h(z) = \sum_{j=0}^\infty b_j z^j$, $|z|<1$. 

For $m \in {\mathbb N}$ and a power series $g(z)=\sum_{j=0}^\infty c_j z^j$ denote by $\Pi^{(m)}(g)$ the polynomial 
\begin{equation}   \label{(e5.7)}
   (\Pi^{(m)}(g))(z) := \sum_{j=0}^{m-1} c_j z^j \,\, , z \in \mathbb C \,.
\end{equation}
From (\ref{(e5.4)}) it is clear that $\Pi^{(m)}(h)$ is defined by $\Pi^{(m)}(f)$ unambiguously. The following lemma implies conversely that $\Pi^{(m)}(h)$ defines $\Pi^{(m)}(f)$ uniquely.

\begin{lemma}   \label{l5.3}
   For any $r \in {\mathbb N}_0$, the $r$-th Fourier coefficient $a_r$ of $\log w$ is uniquely defined by the first $r+1$ Taylor coefficients $b_0, ... ,b_r$ of $h$.
\end{lemma}

\begin{proof}
Relation (\ref{(e5.4)}) can be written as $\exp ({-\frac{a_0}{2}})h(z)=\exp (\sum_{j=1}^\infty a_j z^j)$. Differentiating we can prove by induction that $\exp ({-\frac{a_0}{2}})h'(z) = \sum_{j=1}^\infty j a_j z^{j-1} \exp (\sum_{j=1}^\infty a_j z^j)$ and 
\begin{eqnarray*}
    \exp ({-\frac{a_0}{2}})h^{(r)}(z) & = & \Bigg[ Q_r \left( \sum_{j=1}^\infty ja_j z^{j-1}, ... , \sum_{j=r-1}^\infty j(j-1) ... (j-r+2)a_j z^{j-r+1} \right) \\
                                                     & + & \sum_{j=r}^\infty j (j-1) ... (j-r+1)a_j z^{j-r} \Bigg] \cdot  \exp \left( \sum_{j=1}^\infty a_j z^j \right)
\end{eqnarray*}
for some polynomial $Q_r$ of $(r-1)$ variables, $r \geq 2$. We derive $\exp ({-\frac{a_0}{2}})h(0)=1$, $\exp ({-\frac{a_0}{2}})h'(0)=a_1$, $\exp ({-\frac{a_0}{2}})h^{(r)}(0) = q_r(a_1, ... ,a_{r-1}) + r! a_r$ for some polynomial $q_r$ of $(r-1)$ variables, $r \geq 2$, which yields the assertion by (\ref{(e5.6)}).
\end{proof}

 Kolmogorov \cite[Thm.~2]{[17]} obtained a formula for $d_2(\mu ; S_m)$ in terms of the Fourier coefficients of $\log w$, which was re-discovered and slightly elaborated by Nakazi and Takahashi \cite{[25]}. Kolmogorov's proof makes use of Hilbert space methods and Hardy space theory, whereas the proof in \cite{[25]} is purely function theoretic and in parts can be generalized to arbitrary $p\in (0,\infty)$. Applying Nakazi's and Takahashi's method we derive an expression for $d_p(\mu ; S_m)$ which can be helpful for stability investigations. 

\begin{lemma}   \label{l5.4}
   Let $m\in \mathbb N$ and $p\in (0,\infty)$. If $\log w \in L^1(\lambda)$, then
\begin{equation}   \label{(e5.8)}
   d_p(\mu ; S_m) = \inf \left\{ \int | 1-\tau |^p \cdot | \exp \Pi^{(m)}(f)|^2 \, d\lambda : \tau \in {\mathcal T}(S_m) \right\} \, .
\end{equation}
\end{lemma}

\begin{proof}
Set $g:= f-\Pi^{(m)}(f)$ and note that
\begin{equation}   \label{(e5.9)}
   \exp \left( \frac{2}{p} g \right) = 1+g_1  \, ,
\end{equation}
where $g_1$ has a Taylor series of the form $g_1(z) = \sum_{j=m}^\infty c_j z^j$, $|z|< 1$. The function $\exp \Pi^{(m)} (f)$ is continuous and root-free, which yields
\begin{equation}   \label{(e5.10)}
   c \leq | \exp \Pi^{(m)} (f) | \leq C \,\, , |z|\leq 1 \, ,
\end{equation}
for some positive constants $c$ and $C$. From (\ref{(e5.4)}) we obtain
\begin{equation}  \label{(e5.11)}
   h =\exp \left( {\rm\Pi}^{(m)}(f) \right) \cdot \exp g \, .
\end{equation}
Taking into account (\ref{(e5.9)}), (\ref{(e5.10)}) and (\ref{(e5.11)}) we derive
\begin{equation}   \label{(e5.12)}
   g_1 \in [ {\mathcal T ( S_m)}]_{\lambda,p} =e_m {H}^p \, .
\end{equation}
Since $h$ is outer, (\ref{(e5.11)}) implies that $\exp g$ is outer as well. From (\ref{(e5.10)}) we conclude that the metrics of $L^p(\lambda)$ and of $L^p( | \exp \Pi^{(m)}(f) |^2 \, d\lambda)$ are equivalent. Therefore, 
\begin{eqnarray*}
d_p(\mu , S_m) & = & \inf \left\{ \int \left| \exp \left( \frac{2}{p}g \right) - \tau \cdot \exp  \left( \frac{2}{p}g \right) \right|^p \left| \exp \Pi^{(m)}(f) \right|^2  d\lambda : \tau \in {\mathcal T}(S_m) \right\} \\
                       & = & \inf \left\{ \int |1-\tau|^p |\exp \Pi^{(m)}(f) |^2 \, d\lambda : \tau \in {\mathcal T}(S_m) \right\} 
\end{eqnarray*}
by (\ref{(e5.5)}), (\ref{(e5.9)}) and the outerness of $\exp ( \frac{2}{p} g )$.
\end{proof}

Using Lemma \ref{l5.3} we can give a sufficient condition for (R1) being true in case $S=S_m$. To state the result we introduce a function $\widetilde{w}_n := \min (w_0,w_n)$, $n \in \mathbb N$, and make the following convention: for $a,b \in [0,\infty)$, $a \geq b$, we set $\log \frac{a}{b} := 0$ if $a=0$ and $\log \frac{a}{b}:= \infty$  if $a > 0$ and $b=0$. 

\begin{theorem}   \label{t5.5}
   Let $\rm w^*$-$\lim \mu_n =\mu_0$. If
\begin{equation}   \label{(e5.13)}
   \lim \int \log \frac{w_0}{\widetilde{w}_n} \, d\lambda = 0 \, ,
\end{equation}
relation (R1) is satisfied for $S=S_m$, $m \in \mathbb N$, and $p \in (0,\infty)$.
\end{theorem}

 To prove Theorem \ref{t5.5} we need an auxiliary result.

\begin{lemma}    \label{l5.6}
   If $\rm w^*$-$\lim \mu_n =\mu_0$, if (\ref{(e5.13)}) is satisfied and  
\begin{equation}   \label{(e5.14)}
   \log w_0 \in L^1(\lambda) \, ,
\end{equation}
then
\begin{equation}   \label{(e5.15)}
   \lim \int | \log w_0 - \log w_n | \, d\lambda = 0 \, .
\end{equation}
\end{lemma}

\begin{proof}
By (\ref{(e5.13)}) and (\ref{(e5.14)}), the function $\log \widetilde{w}_n$ belongs to $L^1(\lambda)$ for all $n$ large enough, and
\begin{equation}    \label{(e5.16)}
   \lim \int \log \widetilde{w}_n \, d\lambda = \int \log w_0 \, d\lambda \, .
\end{equation}
Therefore, $\lim\inf \int \log w_n \, d\lambda \geq \int \log w_0 \, d\lambda$. Choosing $S=S_1$ in Theorem \ref{t3.1} and taking into account (\ref{(e5.2)}), we get $\lim\sup\int \log w_n \, d\lambda \leq \int \log w_0\, d\lambda$. Hence, $ \lim \int ( \log w_0 - \log w_n ) \, d\lambda = 0$. Since 
\[
   \int ( \log w_0 - \log w_n ) \, d\lambda = \int ( \log w_0 - \log \widetilde{w}_n )\, d\lambda +  \int_{w_0 < w_n} ( \log w_0 - \log w_n ) \, d\lambda \, , 
\]
relation (\ref{(e5.16)}) yields $\lim \int_{w_0 < w_n} ( \log w_0 - \log w_n ) \, d\lambda =0$. We derive 
\begin{eqnarray*}
   \lim \int | \log w_0 - \log w_n | \, d\lambda & = &  \int ( \log w_0 - \log \widetilde{w}_n )\, d\lambda \\
                                                                   & - &  \lim \int_{w_0 < w_n} ( \log w_0 - \log w_n ) \, d\lambda =0 \, .
\end{eqnarray*}
\end{proof}

{\it Proof of Theorem \ref{t5.5}.} If $\log w_0 \not\in L^1(\lambda)$, then $d_p(\mu_o ; S_m)=0$ and the result follows from Corollary \ref{c3.3}(i). Let $\log w_0 \in L^1(\lambda)$. Relation (\ref{(e5.15)}) implies that $\log w_n \in L^1(\lambda)$ for all suffiently large $n$. Define $f_n$ by $f_n(z):= \frac{1}{2} \int \frac{e_1+z}{e_1-z} \log w_n \, d\lambda$, $|z|<1$, and $\Pi^{(m)}(f_n) $ for the corresponding power series, cf.~(\ref{(e5.3)}) and (\ref{(e5.7)}). From (\ref{(e5.15)}) it follows that the sequence $(| \exp \Pi^{(m)}(f_n) |^2 )_n$ converges to $| \exp \Pi^{(m)}(f_0) |^2$ uniformly on the unit circle. Thus, (\ref{(e5.10)}) yields the existence of a sequence $(c_n)$ of positive numbers such that $\lim\inf c_n \geq 1$ and $| \exp \Pi^{(m)}(f_n) |^2\geq c_n \cdot | \exp \Pi^{(m)}(f_0) |^2$. An application of Lemma \ref{l5.4} and Corollary \ref{c3.3}(ii) completes the proof.
$\Box$ \linebreak[4]

\begin{corollary}   \label{c5.7}
   If $\lim\| \mu_0-\mu_n\| =0$ and $\mu_n \geq c\mu_0$ for some positive constant $c$ and any $n \in \mathbb N$, then (R1) is true for all $S=S_m$ and all $p\in (0,\infty)$.
\end{corollary}

\begin{proof}
According to Corollary \ref{c3.3}(i) we can assume that $\log w_0 \in L^1(\lambda)$. If (R1) would not be true, there would exist $m \in\mathbb N$, $p \in (0,\infty)$, and a subsequence $(w_{n'})$ such that 
\begin{equation}   \label{(e5.17)}
    \lim_{n' \to \infty} d_p(\mu_{n'} ; S_m) \not= d_p(\mu_0; S_m) \, .
\end{equation}
For a subsequence $(n'')$ of $(n')$ with $\lim_{n'' \to \infty} w_{n''} = w_0$ $\lambda$-a.e. we get \linebreak[4]
$\lim_{n'' \to \infty} \int \frac{\log w_0}{\log w_{n''}} \, d\lambda=0$ by Lebesgue's dominated convergence theorem. Hence, $\lim_{n'' \to \infty} d_p(\mu_{n''} ; S_m) = d_p(\mu_0 ; S_m )$ by Theorem \ref{t5.5}, which is a contradiction to (\ref{(e5.17)}).
\end{proof}

The preceeding assertion does not remain true if the norm-convergence of $(\mu_n)$ is replaced by the convergence in measure or with weak convergence as can be seen from Example \ref{e5.2} or from the following example, resp..

\begin{example}   \label{e5.8}
Define $w_k$ as in Example \ref{e4.11}. Then 
\begin{eqnarray*}
d_2(\mu_n ; S_1) & = &  \exp \left( \frac{1}{2\pi}\int \log (2+\cos n\gamma) \, d\gamma \right) \\
& = & \exp \left(\frac{1}{2\pi} \int \log 2 \, d\gamma \right) \cdot \exp \left( \frac{1}{2\pi} \int \log \left(1+\frac{\cos n\gamma}{2} \right) \right) \, d\lambda \\
& = & d_2(\mu_0 ; S_1) \cdot \exp \left(  \frac{1}{2\pi} \int \sum_{j=1}^\infty (-1)^j \frac{1}{j} \left(\frac{\cos \gamma}{2} \right)^j \, d\lambda \right) \\
& \geq & d_2(\mu_0 ; S_1)
\end{eqnarray*}
since the summands with odd indices give $0$.
\end{example}

\begin{remark}  \label{r5.9}
Set $G=\Gamma=\mathbb R$, $a \in (0,\infty)$ and $S:=(-\infty,-a]$. In their paper \cite{[21]} Malevich and Mirzakhmedow investigated the stability of the prediction error $d_2(\mu ; S)$. They claim that if $d\mu_k=w_k \, d\lambda$, $\lim \| \mu_0 - \mu_n \|=0$, $\lim \int \log \frac{|w_0-w_n|}{w_n} \, d\lambda =0$, then $\lim d_2(\mu_n ; S)=d_2(\mu_0 ; S)$. Unfortunately, their proof contains gaps. Note, that our Theorem \ref{t5.5} is a similar, but stronger assertion for the case $G=\mathbb Z$, $\Gamma=[0,2\pi)$, $S=S_m$ since for any $b \in (0,1]$, $\log \frac{w_0}{\widetilde{w}_n} = \log \frac{w_0}{\widetilde{w}_n} -\log 1 \leq \frac{1}{b} \frac{|w_0-\widetilde{w}_n|^b}{\widetilde{w}_n^b} \leq \frac{1}{b} \frac{|w_0-w_n|^b}{w_n^b}$. Moreover, our Theorem \ref{t5.11} below states a result on $d_2(\mu_n,\mu_0; S_m)$, whose analogue in the case $G=\Gamma=\mathbb R$ was also claimed in \cite{[21]}.
\end{remark}

If the metric projection of the function $1$ onto the space $[{\mathcal T}(S_m)]_{\mu,p}$ exists, it is denoted by $\phi^{(m)}_{\mu,p}$. From \cite[Thm.~8.1]{[6]} it can be derived that  $\phi^{(m)}_{\mu,p}$ exists for all $m \in \mathbb N$ and $p \in [1,\infty)$. To study the theoretical error of the estimated prediction it would be helpful to have an expression for  $\phi^{(m)}_{\mu,p}$. To the best of our knowledge Theorem \ref{t5.10} summarizes  the information on  $\phi^{(m)}_{\mu,p}$ known at present. In case $p=2$ the assertion is an old result of prediction theory and can be found e.g.~in \cite{[4]}. It was extended to $p \in (1,2]$ by Cambanis and Soltani \cite[Thm.~5.1]{[3]} and, with a different proof, to $p \in (1, \infty)$ by Rajput and Sundberg in \cite[Thms.~2 and 4]{[27]}. Recently, the result was re-discovered, cf.~\cite{[19]}. 

Let $\log w \in L^1(\lambda)$. Recall the definition (\ref{(e5.4)}) of the corresponding Hardy function $h$ and the definition of $\Pi^{(m)}(h)$, cf.~(\ref{(e5.7)}).

\begin{theorem}    \label{t5.10}
   Let $m \in \mathbb N$ and $p \in [1, \infty)$. If $p=1$ or $p=2$ or if 
\begin{equation}    \label{(e5.18)}
   (\Pi^{(m)}(h))(z) \not= 0 \quad {\rm for} \,\, |z|<1  \,\, ,
\end{equation}
then
\begin{equation}    \label{(e5.19)}
    \phi^{(m)}_{\mu,p} = 1 - \left( \frac{\Pi^{(m)}(h)}{h} \right)^\frac{2}{p}
\end{equation}
and
\begin{equation}     \label{(e5.20)}
   d_p(\mu ; S_m) = \int | \Pi^{(m)}(h) |^2 \, d\lambda = \sum_{j=0}^{m-1} |b_j|^2 \,\, .
\end{equation}
\end{theorem}

Let $\nu$ be an absolutely continuous measure of $\mathcal M$ such that $\mu << \nu$. Set $w_\nu := \frac{d\nu}{d\lambda}$ and let $h_\nu$ be the correspondung Hardy function (\ref{(e5.4)}) if $\log w_\nu \in L^1(\lambda)$. If $\Pi^{(m)}(h_\nu)$ is root-free in the open unit disc, from Theorem \ref{t5.10} it follows
\begin{equation}   \label{(e5.21)}
   d_p(\nu,\mu;S_m) = \int | \Pi^{(m)}(h_\nu)|^2 \frac{w}{w_\nu} \, d\lambda \,\, , m\in {\mathbb N} \, , p \in [1,\infty) \, .
\end{equation}
In particular, since $\Pi^{(1)}(h_\nu)=h_\nu(0)=\exp (\frac{1}{2} \int \log w_\nu \, d\lambda ) \not= 0$, for $m=1$ condition (\ref{(e5.18)}) is always satisfied and
\begin{equation}    \label{(e5.22)}
   d_p(\nu,\mu; S_1)=\int \frac{w}{w_\nu} \, d\lambda \cdot \exp \left( \int \log w_\nu \, d\lambda \right) , p\in [1,\infty) \, .
\end{equation}
Moreover, in case $m=1$ an assertion analogous to Theorem \ref{t4.9} can be proved.

\begin{theorem}   \label{t5.11}
   Let $p \in [1,\infty)$. The error $d_p(\nu,\mu ; S_1)$ is equal to zero or satisfies the inequality 
\begin{equation}   \label{(e5.23)}
   d_p(\mu; S_1) \leq d_p(\nu,\mu ; S_1) \leq \infty  \, .
\end{equation}
If $d_p(\mu ; S_1) >0$, the first inequality of (\ref{(e5.23)}) becomes an equality if and only if $\nu = a \mu$ for some constant $a \in (0,\infty)$. 
\end{theorem}

\begin{proof}
To prove the first assertion we have only to show that if both $\log w$ and $\log w_\nu$ are integrable with respect to $\lambda$, then the left inequality of (\ref{(e5.23)}) is satisfied. According to (\ref{(e5.2)}) and (\ref{(e5.22)}) one has to derive the inequality $\exp \int \log w\, d\lambda \leq \int \frac{w}{w_\nu} \, d\lambda\cdot \exp \int \log w_\nu \, d\lambda$ or, equivalently, $\exp \int \log \frac{w}{w_\nu} \, d\lambda \leq \int \frac{w}{w_\nu} \, d\lambda$, which is Jensen's inequality. The second assertion merely expresses the condition for equality in Jensen's inequality.
\end{proof}

We slightly extend Theorem \ref{t5.10} showing that condition (\ref{(e5.18)}) is not a necessary assumption for the equality (\ref{(e5.20)}).

\begin{corollary}    \label{c5.12}
  For all $m \in \mathbb N$ and $p \in [1,\infty)$, equality (\ref{(e5.20)}) is satisfied.
\end{corollary}

\begin{proof} 
Let $o_m(z):= \sum_{j=0}^\infty c_j z^j$, $z \in \mathbb C$, be the outer part of $\Pi^{(m)}(h)$. Since $|\Pi^{(m)}(h) |=|o_m|$ on the unit circle and since Lemmata \ref{l5.3} and \ref{l5.4}  imply that $d_p(\mu ; S_m)$ is uniquely defined by the first $m$ Taylor coefficients of $h$, it follows $d_p(\mu ; S_m)= d_p(| \Pi^{(m)}(h)|^2 d\lambda ;S_m) = d_p(|o_m|^2 \, d\lambda ; S_m)$. Now, $d_p(|o_m|^2 \, d\lambda ; S_m) =\sum_{j=0}^{m-1} |c_j|^2 = \int | o_m |^2 \, d\lambda = \int |\Pi^{(m)}(h)|^2 \, d\lambda = \sum_ {j=0}^{m-1} | b_j |^2$ by Theorem \ref{t5.5} and Parseval's relation.
\end{proof}

We are now in the position to give sufficient conditions for (R2). If $d\mu_k =w_k d\lambda$ and $\log w_k \in L^1(\lambda)$, denote the corresponding Hardy function by $h_k$, cf.~(\ref{(e5.4)}). 

\begin{theorem}   \label{t5.13}
   Let $m \in \mathbb N$ and $p \in [1,\infty)$. Suppose, the absolutely continuous measures $\mu_n$, $n \in {\mathbb N}_0$, satisfy the following conditions:
\begin{itemize}
\item[(i)] $\rm w^*$-$\lim \mu_n = \mu_0$,
\item[(ii)] $\mu_0 << \mu_n$ and $\lim\inf d_p(\mu_n,\mu_0 ; S_m) \geq d_p(\mu_0 ; S_m)$,
\item[(iii)] $\log w_0 \in L^1(\lambda)$,
\item[(iv)] if $p \not= 1$ and $p \not=2$, then for all sufficiently large $n$, $\Pi^{(m)}(h_n)$ does not have zeros in the open unit disc,
\item[(v)] $\lim \int \frac{|w_0-w_n|}{w_n} \, d\lambda = 0$.
\end{itemize}
Then (R2) is true for $S=S_m$.
\end{theorem}

\begin{proof}
Since relation (\ref{(e5.13)}) is a consequence of condition (v), cf.~Remark \ref{r5.9}, conditions (iii) and (v) yield (\ref{(e5.15)}), which implies that the $j$-th Taylor coefficient of $h_n$ tends to the $j$-th Taylor coefficient of $h_0$ for $n \to \infty$, $j \in {\mathbb N}_0$. Therefore, from (\ref{(e5.19)}) we obtain $|1-\phi^{(m)}_{n,p} |^p \leq \frac{c}{|h_n |^2}$ $\, \lambda$-a.e. for some positive constant $c$ and all $n$ large enough. We derive
\begin{eqnarray*}
    d_p(\mu_n, \mu_0 ; S_m) & = &  \int | 1-\phi^{(m)}_{n,p} |^p w_0 \, d\lambda \\
                                           & \leq & \int  | 1-\phi^{(m)}_{n,p} |^p | w_0-w_n | \, d\lambda + \int | 1-\phi^{(m)}_{n,p} |^p w_n \, d\lambda \\
                                           & \leq & c \cdot \int \frac{|w_0-w_n |}{w_n} \, d\lambda + d_p(\mu_n ; S_m)   , ,
\end{eqnarray*}
hence, $\lim\sup d_p(\mu_n,\mu_0;S_m) \leq d_p(\mu_0 ; S_m)$ by condition (v) and Theorem \ref{t5.5}. To complete the proof take into account condition (ii).
\end{proof}

As the following example reveals, condition (v) of the preceding theorem cannot be replaced by the weaker condition that $\lim \int \frac{|w_0-w_n|^a}{w_n^a} \, d\lambda = 0$ for some $a \in (0,1)$. The more, it cannot be replaced by condition (\ref{(e5.13)}).

\begin{example}  \label{e5.14}
For $a \in (0,1)$ select $b \in (1,\infty)$ such that $ab<1$, and set $w_0=1$, $w_n := n^{-b} \cdot 1_{[0,\frac{1}{n})} + w_0 \cdot 1_{[\frac{1}{n},2\pi)}$. Then $\lim \int |w_0-w_n | \, d\lambda =0$, $\int \frac{|w_0-w_n |^a}{w_n^a} \, d\lambda= \lim \frac{(n^b-1)^a}{n} =0$, $d_2(\mu_0 ; S_1) =1$, and by (\ref{(e5.22)}) 
\begin{eqnarray*}
d_2(\mu_n,\mu_0 ; S_1) & = & \int \frac{w_0}{w_n} \, d\lambda \cdot \exp \left( \int \log w_n \, d\lambda \right) \\
                                     & = & \frac{1}{2\pi} \left( n^{b-1} + 2\pi -\frac{1}{n} \right) \cdot \exp \left( -\frac{b \log n}{n} \right) \, ,
\end{eqnarray*}
 which tends to $\infty$ if $n$ tends to $\infty$.
\end{example}

From Theorem \ref{t5.10} a function theoretic inequality can be derived, which does not seem to be proven so easy with purely function theoretic means.

\begin{proposition}   \label{p5.15}
   Let $h,g \in H^2$ and $g$ be an outer function. Then for all $m\in \mathbb N$,
\begin{equation}   \label{(e5.24)}
    \int \left| \Pi^{(m)} (g) \right|^2 \frac{|h|^2}{|g|^2} \, d\lambda \geq \sum_{j=0}^{m-1} |b_j |^2 \, ,
\end{equation}
where $(b_j)$ is defined by (\ref{(e5.6)}).
\end{proposition}

\begin{proof}
Let $w=|h|^2$, $v=|g|^2$, $d\mu = w \, d\lambda$, $d\nu =v \, d\lambda$, $w_n:= \min (w, nv)$, $d\mu_n = w_n \, d\lambda$. First, assume that $h$ is an outer function. Since 
\begin{eqnarray*}
    \int |\log w_n | \, d\lambda & = & \int_{w \leq nv} | \log w | \, d\lambda + \int_{nv < w} | \log nv | \, d\lambda \\
                                         & \leq & \int | \log w | \, d\lambda + \log n + \int | \log v |\, d\lambda \\
                                         & < & \infty,
\end{eqnarray*}
there exists an outer function $h_n \in H^2$ such that $w_n=|h_n|^2$ $\, \lambda$-a.e. on the unit circle. Set $b_{j,n} := \frac{h_n^{(j)} (0)}{j!}$. Since $w < \infty$ and $v > 0$ $\, \lambda$-a.e., the sequence $(1_{nv < w} (\log w - \log w_n)$ is decreasing and tends to zero for $n\to \infty$. Thus, $\lim \int | \log w - \log w_n | \, d\lambda = \lim \int 1_{nv<w} (\log w - \log w_n) \, d\lambda =0$ by the monotone convergence theorem, which yields 
\begin{equation}     \label{(e5.25)}
   \lim  b_{j,n} = b_j, j \in {\mathbb N}_0 \, .
\end{equation}
The inequality $\mu_n \leq n \nu$ implies $\phi_{\nu,2}^{(m)} \in [{\mathcal T}(S_m)]_{\mu_n,2}$, hence, $\int |\Pi^{(m)}(g)|^2 \frac{|h_n|^2}{|g|^2} \, d\lambda = d_2(\nu,\mu_n ; S_m) \geq d_2(\mu_n ; S_m) = \sum_{j=0}^{m-1} |b_{j,n}|^2$ by (\ref{(e5.21)}) and (\ref{(e5.20)}). To obtain (\ref{(e5.24)}), let $n$ tend to $\infty$ and apply the monotone convergence theorem and (\ref{(e5.25)}). 

If $h$ is not outer and $\widetilde{h}$ denotes its outer part, then $|h|=|\widetilde{h}|$ $\, \lambda$-a.e. on the unit circle, $\sum_{j=1}^{m-1} \left| \frac{\widetilde{h}^{(j)}(0)}{j!} \right|^2 \geq \sum_{j=0}^{m-1}| b_j |^2$ and (\ref{(e5.24)}) follows from the result just proved.
\end{proof}

If $h=1$ and $g(z)=z$ or $g(z)=\exp \left(- \frac{1+z}{1-z} \right)$, $|z|<1$, then $b_0=1$ and $\int | \Pi^{(1)}(g)|^2 \frac{|h|^2}{|g|^2} \, d\lambda$ equals to $0$ or to ${\rm e}^{-2}$, respectively. Consequently, the outerness of $g$ is a necessary supposition in Proposition \ref{p5.15}.

We conclude this section with an extension of Theorem \ref{t5.11} from $m=1$ to arbitrary $m > 0$, what  is an immediate consequence of (\ref{(e5.20)}), (\ref{(e5.21)}) and (\ref{(e5.24)}).  

\begin{corollary}   \label{c5.16}
   Let $m \in\mathbb N$ and $p \in [1, \infty)$. Let $\mu, \nu \in \mathcal M$ be measures of the form $d\mu = |h|^2 \, d\lambda$, $d\nu = |g|^2 \, d\lambda$ for $h,g \in H^2$, $g$ outer. If $d_p(\nu,\mu ; S_m) = \int | \Pi^{(m)}(g) |^2 \frac{|h|^2}{|g|^2} \, d\lambda$, then $d_p(\mu ; S_m) \leq d_p(\nu, \mu ; S_m)$. 
\end{corollary}

%%%%%%%%%%%%%%%%%%%%%%%%%%%

\section{Finite set of observations}

In this section $G$ is an arbitrary LCA group and $S_k:= \{ x_1, ... ,x_k \}$, $k\in \mathbb N$, is a finite subset of $G \setminus \{ 0 \}$. Write ${\mathcal T}:= {\mathcal T}(S_k)$. For $\mu \in \mathcal M$ and $p \in (0,\infty)$, the space $[{\mathcal T}]_{\mu,p} = \mathcal T$ is a finite-dimensional linear space not depending on $p$. However, note that its dimension can be less than $k$ since $\mu$-equivalent functions are identified. To emphasize that we are concerned with $\mu$-equivalence classes of functions we use the notation $[{\mathcal T}]_\mu$ instead of $\mathcal T$. 

Let $\nu \in \mathcal M$ satisfy $\mu << \nu$ and let $p \in (0,\infty)$ be such that the metric projection $\phi_{\nu,p}$ exists. Since $\phi_{\nu,p} \in[{\mathcal T}]_\mu$ we have
\begin{equation}   \label{(e6.1)}
   d_p(\mu) = \inf \{ \int|1-\tau |^p \, d\mu : \tau \in [{\mathcal T}]_\mu \} \leq \int | 1-\phi_{\nu,p} |^p \, d\mu = d_p(\nu,\mu) \, .
\end{equation}
We recall some facts on subadditive functionals  which are absolutely homogeneous of order $p$. Let $L$ be a finite-dimensional linear space over $\mathbb C$ equipped with a norm $\| \cdot \|_L$. A functional $f: L\to [0,\infty)$ is called subadditive if $f(u+v) \leq f(u)+f(v)$, $u,v \in L$. It is absolutely homogeneous of order $p$ if $f(au)=|a|^p f(u)$, $a \in \mathbb C$, $u \in L$. Denote the set of all subadditive and absolutely homogeneous of order $p$ functionals on $L$ by ${\mathcal F}_L$. Note, that subadditivity yields 
\begin{equation}   \label{(e6.2)}
   | f(u)-f(v) | \leq f(u-v) \, .
\end{equation}
Moreover, if $f \in {\mathcal F}_L$ then $f$ is continuous and there exists a constant $C$ with
\begin{equation}  \label{(e6.3)}
   f(u) \leq C \cdot \| u \|_L \,\, , u\in L\, .
\end{equation}
If additionally $f(u)=0$ only for $u=0$, then $\rho (u,v):= f(u-v)$ defines a metric on $L$, and there exists a constant $c \in (0,\infty)$ with 
\begin{equation}   \label{(e6.4)}
   c \cdot \|u\|_L \leq f(u) \,\, , u \in L \, .
\end{equation}
From (\ref{(e6.3)}) and (\ref{(e6.4)}) we derive that the topology generated by the metric $\rho$ is equivalent to the norm topology on $L$ and that the set $K:= \{ u \in L: f(u)=1 \}$ is compact. 

Let $(f_n)$ be a sequence of functionals of ${\mathcal F}_L$. If $f(u):= \sup \{ f_n(u) : n \in {\mathbb N} \} < \infty$ for all $u\in L$, then $f \in {\mathcal F}_L$. If $f(u):= \lim f_n(u)$ exists for any $u \in L$, then $f \in {\mathcal F}_L$.

\begin{lemma}   \label{l6.1}
   If the sequence $(f_n)$ converges to $f_0 \in {\mathcal F}_L$ pointwise and $f_0(u) > 0$ for all $u \in L \setminus  \{ 0 \}$, then the convergence is uniform on the set $K_0 := \{u \in L: f_0(u)=1 \}$.
\end{lemma}

\begin{proof} 
Since the sequence $(f^{(j)})_{j \in \mathbb N}$, $f^{(j)}:= \sup \{ f_n : n \geq j \}$ is increasing and converges to $f_0$ pointwise, it converges to $f_0$ uniformly on $K_0$ by Dini's theorem, in particular, $f^{(j)}(u)-f_0(u) \leq b$ for some $b \in (0,\infty)$ and all $u \in K_0$ and $j \in \mathbb N$. If $v \in L \setminus \{ 0 \}$ then $(f_0(v))^{-\frac{1}{p}} v \in K_0$, which yields $f^{(j)}(v)(f_0(v))^{-1} - 1 \leq b$, hence, 
\begin{equation}   \label{(e6.5)}
   f_j(v) \leq f^{(j)}(v) \leq (b+1)f_0(v) \,\, , j \in \mathbb N \, .
\end{equation}
Let $\varepsilon \in (0,\infty)$. Choose a finite subset $D$ of $K_0$ such that for any $v \in K_0$ there exists $u_v \in D$ with
\begin{equation}   \label{(e6.6)}
    f_0(v-u_v) < \frac{2\varepsilon}{3(b+2)} \,\, .
\end{equation}
Now, choose $n_0 \in \mathbb N$ such that
\begin{equation}    \label{(e6.7)}
   | f_0(u)-f_n(u) | < \frac{\varepsilon}{3}
\end{equation}
for all $u \in D$ and any $n \geq n_0$. From (\ref{(e6.2)}),  (\ref{(e6.5)}), (\ref{(e6.6)}) and (\ref{(e6.7)}) the inequality 
\begin{eqnarray*}
    | f_0(v) - f_n(v) | & \leq & | f_0(u)-f_0(u_v) | + | f_0(u_v)-f_n(u_v) | + | f_n(u_v) - f_n(v) |\\
                              & < & \frac{2\varepsilon}{3(b+2)} + \frac{\varepsilon}{3} + \frac{2\epsilon}{3(b+2)} (b+1) \\
                              & = & \varepsilon
\end{eqnarray*}
follows for $v \in K_0 \setminus D$, $n \geq n_0$.
\end{proof}

\begin{theorem}   \label{t6.2}
Let $\rm w^*$-$\lim \mu_n = \mu_0.$
\begin{itemize}
\item[(i)]   Let $p\in (0,\infty)$. If $d_p(\mu_0 ; S_k) =0$ or $[{\mathcal T}(S_k)]_{\mu_0}$ is a $k$-dimensional subspace of $L^p(\mu_0)$, then (R1) is true for $S=S_k$. 
\item[(ii)]   Let $p \in (1,\infty)$. If ${\mathcal T}(S_k \cup \{ 0 \})]_{\mu_0}$ is a $(k+1)$-dimensional subspace of $L^p(\mu_0)$, then (R2) and (R3) are satisfied for $S=S_k$.
\end{itemize}
\end{theorem}

\begin{proof}
(i) According to Corollary \ref{c3.3}(i) we have only to consider the case $d_p(\mu_0) >0$. Let $p \in [1,\infty)$ and set $L:= [{\mathcal T}(S_k \cup \{ 0 \}) ]_{\mu_0}$. Define $f_j \in {\mathcal F}_L$ by $f_j(\tau):= (\int |\tau |^p \, d\mu_j)^\frac{1}{p}$, $\tau \in L$, $j \in {\mathbb N}_0$. Note, $f_0(\tau)>0$ for all $\tau \in L \setminus \{ 0 \}$. Let $\varepsilon \in (0,\infty)$. By Lemma \ref{l6.1} there exists $n_0 \in \mathbb N$ such that $f_0(\tau)-\varepsilon \leq f_n(\tau)$ for all $\tau \in K_0$ and all $n \geq n_0$. Therefore,
\begin{equation}   \label{(e6.8)}
    (1-\varepsilon)^p \cdot \int | \tau |^p \, d\mu_0 < \int | \tau |^p \, d\mu_n
\end{equation}
for all $\tau \in L$ and all $n \geq n_0$. Taking the infimum over all $\tau$ of the form $\tau=1-\widetilde{\tau}$, $\widetilde{\tau} \in [{\mathcal T}(S_k)]_{\mu_0}$, we get $(1-\varepsilon)^p d_p(\mu_0) \leq d_p(\mu_n)$ for all $n \geq n_0$, which implies that $d_p(\mu_0) \leq \lim\inf d_p(\mu_n)$ since $\varepsilon \in (0,\infty)$ was arbitrary. Taking into account Theorem \ref{t3.1}, the proof is completed for $p \in [1,\infty)$. In the case $p \in (0,1)$ a similar proof works applying Lemma \ref{l6.1} to the functionals $f_j(\tau) := \int | \tau |^p \, d\mu_j$ with $j \in {\mathbb N}$.

(ii) Set $\tau:= \phi_{n,p}$ in (\ref{(e6.8)}), let $\varepsilon$ tend to $0$ and $n$ tend to $\infty$. We obtain $\lim \sup d_p(\mu_n,\mu_0) \leq \lim d_p(\mu_n) = d_p(\mu_0)$ by (i). Now, (R2) follows from (\ref{(e6.1)}) and, eventually, (R3) from Theorem \ref{t3.9}.
\end{proof}

\begin{corollary}   \label{c6.3}
   If $\lim \| \mu_0-\mu_n \|=0$, then (R1) is true for $S=S_k$ and $p \in (0,\infty)$. 
\end{corollary}

\begin{proof}
By Corollary \ref{c3.4} one has only to prove (R1) for a sequence $(\mu_n)$ satisfying $\lim \|\mu_0-\mu_n \|=0$ and $\mu_n \leq \mu_0$, $n \in \mathbb N$. Assume without loss of generality that the family of functions $e_{x_1}, ... ,e_{x_j}$, $j \leq k$, is a basis of $[{\mathcal T}]_{\mu_0}$. Since $\mu_n \leq \mu_0$ these functions span $[{\mathcal T}]_{\mu_n}$ and the claim is proved by Theorem \ref{t6.2}(i).
\end{proof}

If $p=2$ and $\lim\| \mu_0-\mu_n\| =0$, the condition of Theorem \ref{t6.2}(ii) can be weakened slightly. Note, that for a singleton $S=S_1$ the conditions (i) and (ii) of Proposition \ref{p3.14} are satisfied. Thus, we can state the following corollary.

\begin{corollary}    \label{c6.4}
   If $\lim\| \mu_0-\mu_n \| =0$ and the space $[{\mathcal T}(S_k)]_{\mu_0}$ is $k$-dimensional, then (R3) is satisfied for $S=S_k$ and $p=2$.
\end{corollary}

As the following example shows the condition w*-$\lim \mu_n = \mu_0$ of Theorem \ref{t6.2} cannot be replaced by m-$\lim \mu_n = \mu_0$. 

\begin{example}   \label{e6.5}
Let $G=\mathbb Z$, $\Gamma = [0,2\pi)$, $S_1=\{ 1 \}$. Define $d\mu_n=(n1_{[0,\frac{1}{n})}+ 1_{[\frac{1}{n},2\pi)}) \, d\lambda$, $n \in \mathbb N$. Then w*-$\lim \mu_n = \frac{1}{2\pi} \delta_0 + \lambda=: \mu_0$ and m-$\lim \mu_n = \lambda$. Since $\int |1-ae_1|^2 \, d\mu_0 = \frac{1}{2\pi} |1-a|^2 +1 +|a|^2$, it is not hard to see that $\min \{ \int | 1-ae_1|^2 \, d\mu_0 : a \in {\mathbb C}\}$ is attained for $a=(1+2\pi)^{-1}$ and $d_2(\mu_0)=(2+2\pi)(1+2\pi)^{-1}$. By Theorem \ref{t6.2}(i), $\lim d_2(\mu_n) = d_2(\mu_0)$, which does not coincide with $d_2(\lambda)=1$. Moreover, if $\phi_{n,2}=1-a_ne_1$, a simple computation reveals that $\lim a_n =(1+2\pi)^{-1}$, which implies 
\begin{eqnarray*}
    \lim d_2(\mu_n, \lambda) & = & \lim \int |1-a_ne_1|^2 \, d\lambda = \lim (1+|a_n|^2) \\
                                          & = & 1+(1+2\pi)^{-2} > 1 = d_2(\lambda) \,\, .
\end{eqnarray*}
Note also, $\lim d_2(\lambda, \mu_n) = \lim \int 1 \, d\mu_n = \frac{1+2\pi}{2\pi} > 1=d_2(\lambda)$.
\end{example}

%%%%%%%%%%%%%%%%%%%%%%%%%%%%%%%%%%%%%%%%%%%

\section{Periodic observations}

This section is devoted to a rather incomplete discussion of the observation set $S=x+H$, where $H$ is a closed subgroup of an LCA group $G$ and $x$ is a given element of $G \setminus H$. Since little is known for $p \not= 2$ most of our results pertain the stationary case $p=2$. Let $A:=\{ \gamma \in \Gamma : e_y(\gamma)=1 \,\,\ {\rm for} \,\, {\rm all} \, y \in H \}$ be the annihilator of $H$. Recall that $A$ is a closed subgroup of $\Gamma$, and thus, an LCA group with  respect to the induced topology. To the end of this section we assume that $A$ is at most countable and denote the number of its elements by ${\rm card} A$. It follows that $A$ is discrete, hence, metrizable. 

A subset $T$ of $\Gamma$ is called a transversal (with respect to $A$) if it meets each $A$-coset just once, equivalently, if $\cup_{\alpha \in A} (\alpha +T)=\Gamma$ and $T \cap(\alpha +T)=\emptyset$ for all $\alpha \in A \setminus \{ 0 \}$, cf.~\cite[Lemma 3.3]{[22]}. According to \cite[Thm.~1]{[7]} we can and shall suppose $T \in \mathcal B$. 

For a function $f: T \to \mathbb C$ define a function $V_xf :\Gamma \to \mathbb C$ by
\begin{equation*}
    (V_xf)(\gamma) := e_x(\alpha) f(\gamma - \alpha)  \quad {\rm for} \,\, \gamma \in (\alpha+T), \, \alpha \in A \, .
\end{equation*}
If $\mu \in \mathcal M$ and $\alpha \in A$, let $\mu^{(\alpha)}$ be a measure on ${\mathcal B}(T)$ satisfying $\mu^{(\alpha)} (B)=\mu(\alpha + B)$ for $B \in {\mathcal B}(T)$, and let $\widetilde{\mu}$ be the regular finite measure $\widetilde{\mu} := \sum_\alpha \mu^{(\alpha)}$, where $\sum_\alpha$ means the sum over all elements $\alpha$ of $A$. 

\begin{lemma}    \label{l7.1}
(cf.~\cite[Lemma 2.1]{[14]})  For $p \in (0,\infty)$, the linear operator $V_x$ establishes an isometric isomorphism between $L^p(\widetilde{\mu})$ and $[{\mathcal T}(x+H)]_{\mu,p}$ satisfying $V_x \tau = \tau$ for all $\tau \in {\mathcal T(x+H)}$ and $V_x^{-1}f=f \cdot 1_T$ for all $f \in [{\mathcal T}(x+H)]_{\mu,p}$.
\end{lemma}

Choose a Radon-Nikod\'ym derivative $h_\mu^{(\alpha)} := \frac{d\mu^{(\alpha)}}{d\widetilde{\mu}}$ and note that 
\begin{equation}   \label{(e7.1)}
   \sum_\alpha h_\mu^{(\alpha)} = 1 \quad \widetilde{\mu}{\rm -} \, {\rm a.e.}
\end{equation}
For $y \in G$ set
\begin{equation}   \label{(e7.2)}
   \widetilde{\phi}_\mu^{(y)} := e_y \sum_\alpha e_{y-x}(\alpha) h_\mu^{(\alpha)}
\end{equation}
and denote $\widetilde{\phi}_\mu := \widetilde{\phi}_\mu^{(0)}$. 

\begin{theorem}   \label{t7.2}
   The following assertions are true:
\begin{itemize}
\item[(i)]  For any $y \in G$ the orthogonal projection $\phi_{\mu,2}^{(y)}$ of $e_y$ onto $[{\mathcal T}(x+H)]_{\mu,2}$ is equal to $V_x \widetilde{\phi}_\mu^{(y)}$.

\item[(ii)] The distance $d_2(\mu ; x+H)$ equals to
\begin{equation}   \label{(e7.3)}
    d_2(\mu ; x+H) = \int_T (1-|\widetilde{\phi}_\mu|^2) \, d\widetilde{\mu} = \int_T \sum_{\alpha, \beta \in A} (1- e_x(\beta-\alpha))h_\mu^{(\alpha)}h_\mu^{(\beta)} \, d\widetilde{\mu} \, .
\end{equation}

\item[(iii)]  If ${\rm card} A=2$ and $A=\{ 0, \alpha \}$, then for $p\in (0,\infty)$ the restriction of the metric projection $\phi_{\mu,p}$ to $T$ is equal to
\begin{equation}    \label{(e7.4)}
    \widetilde{\psi} := \left( \left( h_\mu^{(0)} \right)^\frac{1}{p-1} - \left(h_\mu^{(\alpha)} \right)^\frac{1}{p-1} \right) \cdot \left( \left( h_\mu^{(0)} \right)^\frac{1}{p-1} + \left( h_\mu^{(\alpha)} \right)^\frac{1}{p-1} \right) \quad \widetilde{\mu} \, {\rm -} \, {\rm a.e.}
\end{equation}
and $\phi_{\mu,p}(\gamma) = -\widetilde{\psi}(\gamma-\alpha)$, $\gamma \in \alpha + T$. The distance $d_p(\mu ; x+H)$ equals to $2^p \cdot \int_T \Big( \left( h_\mu^{(0)} \right)^\frac{1}{p-1} + \left( h_\mu^{(\alpha)} \right)^\frac{1}{p-1} \Big) h_\mu^{(0)}h_\mu^{(\alpha)} \, d\widetilde{\mu}$.
\end{itemize}
\end{theorem}

\begin{proof}
(i) By Lemma \ref{l7.1} the orthogonal projection of $e_y$ onto $[{\mathcal T}]_{\mu,2}$has the form $V_xg$ for some $g \in L^2(\widetilde{\mu})$. From Beppo Levi's theorem and (\ref{(e7.1)}) it follows
\begin{equation*}
\begin{split}
  \int | e_y-V_xg |^2 \, d\mu & =   \sum_\alpha \int_{\alpha+T} | e_y(\gamma)-e_x(\alpha)g(\gamma-\alpha) |^2 \mu(d\gamma) \\
                                    & =  \sum_\alpha \int_T | e_y(\alpha)e_y - e_x(\alpha)g |^2 h_\mu^{(\alpha)} \, d\widetilde{\mu} \\
                                    & =  \int_T \sum_\alpha |e_y -e_{x-y}(\alpha)g |^2 h_\mu^{(\alpha)} \, d\widetilde{\mu} \\
                                    & =  \int_T \sum_\alpha | e_y |^2 h_\mu^{(\alpha)} \, d\widetilde{\mu} - \int_T \bigg| \sum_\alpha e_{y-x}(\alpha) h_\mu^{(\alpha)} \bigg|^2 \, d\widetilde{\mu}\\
                                    &  \quad + \int_T \bigg| e_y \cdot \sum_\alpha e_{y-x}(\alpha) h_\mu^{(\alpha)} - g \bigg|^2 \, d\widetilde{\mu}
\end{split}
\end{equation*}
which becomes minimal if and only if the third integral in the sum at the end equals to $0$, i.e. $g=\widetilde{\phi}_\mu^{(y)}$. 

(ii) Note, that 
\begin{eqnarray*}
    d_2(\mu) & = & \int |1-\phi_{\mu,2}|^2 \, d\mu \\
                  & = & \int (1-|\phi_{\mu,2}|^2)\, d\mu \\
                  & = & \int ( | V_x1|^2 - | V_x\widetilde{\phi}_\mu |^2) \, d\mu \\
                  & = & \int_T (1-| \widetilde{\phi}_\mu |^2)\, d\widetilde{\mu} \\
                  & = & \int_T \bigg(\big| \sum_\alpha h_\mu^{(\alpha)} \big|^2 - \big| \sum_\alpha e_x(-\alpha) h_\mu^{(\alpha)} \big|^2 \bigg) \, d\widetilde{\mu}
\end{eqnarray*}
by Lemma \ref{l7.1} and (\ref{(e7.1)}). Write the sums out. 

(iii) The proof of (iii) is a straightforward adaption of the proof of the particular case $G=\mathbb Z$, $H=2\mathbb Z$, $\Gamma = [0,2\pi)$, $x=1$, $A= \{ 0, \pi \}$ given in \cite[Thm.~5.2]{[15]}. It is, thus, omitted.
\end{proof}

For $\mu_k, \nu \in \mathcal M$, $\mu_k << \nu$ set $w_k:=\frac{d\mu_k}{d\nu}$, $w_k^{(\alpha)} := \frac{d\mu_k^{(\alpha)}}{d\nu_k^{(\alpha)}}$, $\widetilde{w}_k := \frac{d\widetilde{\mu}_k}{d\widetilde{\nu}}$, $k \in {\mathbb N}_0$, $\alpha \in A$. Write $h_k^{(\alpha)} := h_{\mu_k}^{(\alpha)}$ and $\widetilde{\phi}_k^{(y)} := \widetilde{\phi}_{\mu_k}^{(y)}$ for short.

\begin{lemma}   \label{l7.3}
   Let $\rm m$-$\lim \mu_n = \mu_0$. If ${\rm card}A < \infty$ or $\mu_n \leq \mu_0$ for all $n \in \mathbb N$, then
\begin{equation}   \label{(e7.5)}
   \widetilde{\nu} {\rm }-\lim \widetilde{w}_n = \widetilde{w}_0 \,\, ,
\end{equation}
\begin{equation}   \label{(e7.6)}
   \widetilde{\nu}{\rm -}\lim \widetilde{\phi}_n^{(y)} \widetilde{w}_n = \widetilde{\phi}_0^{(y)} \widetilde{w}_0 \,\, ,
\end{equation}
and
\begin{equation}   \label{(e7.7)}
   \widetilde{\mu}_0{\rm -}\lim \widetilde{\phi}_n^{(y)} = \widetilde{\phi}_0^{(y)} \,\, ,
\end{equation}
for all $y \in G$.
\end{lemma}

\begin{proof}
Relation (\ref{(e7.5)}) follows from (\ref{(e7.6)}) by (\ref{(e7.1)}) and (\ref{(e7.2)}). Also, (\ref{(e7.6)}) yields $\widetilde{\nu}$-$\lim\widetilde{\phi}_n^{(y)}\widetilde{w}_n  1_B = \widetilde{\phi}_0^{(y)}\widetilde{w}_0 1_B$, where $B:= \{ \gamma \in T : \widetilde{w}_0(\gamma) > 0 \}$. Since $\mu_0(T \setminus B)=0$ and $\widetilde{\mu}_0$ and $\widetilde{\nu}$ are equivalent on $B$, (\ref{(e7.7)}) follows, cf.~\cite[Section 23.1]{[9]}. So we have only to prove (\ref{(e7.6)}). First note, that 
\begin{equation}   \label{(e7.8)}
   h_k^{(\alpha)} \widetilde{w}_k = w_k^{(\alpha)} h_\nu^{(\alpha)},\, k \in {\mathbb N}_0 \, ,
\end{equation}
since $\frac{d\mu_k^{(\alpha)}}{d\widetilde{\mu}_k} \cdot \frac{d\widetilde{\mu}_k}{d\widetilde{\nu}} = \frac{d\mu_k^{(\alpha)}}{d\widetilde{\nu}} = \frac{d\mu_k^{(\alpha)}}{d\nu^{(\alpha)}} \cdot \frac{d\nu^{(\alpha)}}{d\widetilde{\nu}}$ . It is easy to see that the condition m-$\lim \mu_n =\mu_0$ and relation (\ref{(e7.8)}) imply that 
\begin{equation}   \label{(e7.9)}
   \widetilde{\nu}-\lim h_n^{(\alpha)} \widetilde{w}_n = h_0^{(\alpha)} \widetilde{w}_0,\, \alpha \in A \, ,
\end{equation}
which yields (\ref{(e7.6)}) if ${\rm card}A < \infty$. Finally, let $A:= \{ \alpha_j : j \in {\mathbb N} \}$ be an infinite set and $\mu_n \leq \mu_0$, $n \in \mathbb N$. We obtain
\begin{equation}   \label{(e7.10)}
   h_n^{(\alpha)} \widetilde{w}_n \leq h_0^{(\alpha)} \widetilde{w}_0 \,\, , n \in {\mathbb N} \, , \alpha \in A \, .
\end{equation}
Since by (\ref{(e7.1)}) the series $\sum_{j=1}^\infty h_0^{(\alpha_j)} \widetilde{w}_0$ converges $\widetilde{\nu}$-a.e., for $\varepsilon \in (0,\infty)$ there exists a number $m\in \mathbb N$ satisfying $\widetilde{\nu} \big( \sum_{j=m+1}^\infty h_0^{(\alpha_j)}\widetilde{w}_0 > \frac{\varepsilon}{2} \big) < \frac{\varepsilon}{2}$. Choose $n_0 \in {\mathbb N}$ such that $\widetilde{\nu} \big(\big| \sum_{j=1}^m h_0^{(\alpha_j)} \widetilde{w}_0 - \sum_{j=1}^m h_n^{(\alpha_j)} \widetilde{w}_n \big| > \frac{\varepsilon}{2} \big) < \frac{\varepsilon}{2}$ for all $n \geq n_0$ and obtain 
\begin{eqnarray*}
    \widetilde{\nu} \left( \left| \widetilde{\phi}_0^{(y)} \widetilde{w}_0 - \widetilde{\phi}_n^{(y)} \widetilde{w}_n \right| > \varepsilon \right)  
& = &  \widetilde{\nu} \bigg( \bigg| \sum_{j=1}^\infty e_{y-x}(\alpha_j)h_0^{(\alpha_j)} \widetilde{w}_0 - \sum_{j=1}^\infty e_{y-x}(\alpha_j) h_n^{(\alpha_j)} \widetilde{w}_n \bigg| > \varepsilon \bigg) \\
& \leq & \widetilde{\nu} \bigg( \sum_{j=1}^\infty h_0^{(\alpha_j)}\widetilde{w}_0 - \sum_{j=1}^\infty h_n^{(\alpha_j)} \widetilde{w}_n > \varepsilon \bigg) \\
& \leq & \widetilde{\nu} \bigg( \sum_{j=1}^m h_0^{(\alpha_j)}\widetilde{w}_0 - \sum_{j=1}^m h_n^{(\alpha_j)} \widetilde{w}_n > \frac{\varepsilon}{2} \bigg) \\
& \, & \quad + \, \, \widetilde{\nu} \bigg( \sum_{j=m+1}^\infty h_0^{(\alpha_j)}\widetilde{w}_0 - \sum_{j=m+1}^\infty h_n^{(\alpha_j)} \widetilde{w}_n > \frac{\varepsilon}{2} \bigg) \\
& < & \frac{\varepsilon}{2} + \widetilde{\nu} \bigg( \sum_{j=m+1}^\infty h_0^{(\alpha_j)} \widetilde{w}_0 > \frac{\varepsilon}{2} \bigg) \\
& < & \varepsilon \,\, ,
\end{eqnarray*}
for all $n \geq n_0$, by (\ref{(e7.2)}) and (\ref{(e7.10)}).
\end{proof}

\begin{theorem}   \label{t7.4}
(i) If $\rm m$-$\lim \mu_n= \mu_0$, then 
\begin{equation}  \label{(e7.11)}
   \lim\inf d_2(\mu_n ; x+H) \geq d_2(\mu_0 ; x+H) \, .
\end{equation}
(ii) If $\rm m$-$\lim \mu_n = \mu_0$ and ${\rm card}A=2$, then $\lim\inf d_p(\mu_n ; x+H) \geq d_p(\mu_0 ; x+H)$ for any $p \in (1,\infty)$.
\end{theorem}

\begin{proof}
(i) Let $\mu_n \leq \mu_0$, $n \in\mathbb N$. Choosing $y=0$ in (\ref{(e7.7)}) we get $\widetilde{\mu}$-$\lim \, (1-|\widetilde{\phi}_n|^2) = (1-|\widetilde{\phi}_0|^2)$ by properties of the convergence in measure, cf.~\cite[Thm.~2.3.3]{[20]}. We get $\widetilde{\nu}$-$\lim \, (1-|\widetilde{\phi}_n|^2) \widetilde{w}_n =(1-|\widetilde{\phi}_0|^2) \widetilde{w}_0$, and hence, 
\begin{eqnarray*}
\lim \inf d_2(\mu_n) & = & \lim\inf \int_T (1-|\widetilde{\phi}_n|^2) \, d\widetilde{\mu}_n \\
        & = &  \lim\inf \int_T (1-|\widetilde{\phi}_n|^2) \widetilde{w}_n \, d\widetilde{\nu} \\
        & \geq &  \int_T (1-|\widetilde{\phi}_0|^2) \widetilde{w}_0 \, d\widetilde{\nu} \\
        & = &  d_2(\mu_0)
\end{eqnarray*}
by (\ref{(e7.3)}) and Fatou's lemma. If the inequality $\mu_n \leq \mu_0$ is not satisfied, define $d\mu_n' := \min (w_0,w_n) \, d\nu$ and obtain $\lim \inf d_2(\mu_n) \geq \lim \inf d_2(\mu_n') \geq d_2(\mu_0)$ by the result just proved. 

(ii) Since elementary inequalities and (\ref{(e7.1)}) yield $\big( ( h_n^{(0)})^\frac{1}{p-1} + (h_n^{(\alpha)})^\frac{1}{p-1} \big)h_n^{(0)} h_n^{(\alpha)} \leq 2^\frac{p-2}{p-1}$, the assertion follows from Theorem \ref{t7.2}(iii) and another application of Fatou's lemma.
\end{proof}

The following example shows that even in the simplest non-trivial case ${\rm card}A=2$ inequality (\ref{(e7.11)}) can be sharp, and that (\ref{(e7.11)}) can fail if convergence in measure is replaced by weak convergence.

\begin{example}   \label{e7.5}
Let $G=\mathbb Z$, $H=2\mathbb Z$, $x=1$. The annihilator $A$ of $H$ is the subgroup $\{0,\pi\}$ of $\Gamma = [0,2\pi)$. Set $T=[0,\pi)$. 

(a) Define $\nu :=\lambda$, $d\mu_n := w_n \, d\lambda$, $w_n:=n  \big( 1_{[0,\frac{1}{n})} + 1_{[\pi,\pi+\frac{1}{n})} \big) + 1_{[\frac{1}{n},\pi)}$, $n \in \mathbb N$. Obviously, m-$\lim \mu_n = \widetilde{\lambda} =: \mu_0$, where $\widetilde{\lambda}$ denotes the restriction of $\lambda$ to $[0,\pi)$. Moreover, $\widetilde{w}_n =2n \cdot 1_{[0,\frac{1}{n})} + 1_{[\frac{1}{n},\pi)}$, $h_n^{(0)} = \frac{1}{2} \cdot 1_{[0,\frac{1}{n})}+ 1_{[\frac{1}{n},\pi)}$, $h_n^{(\pi)} = \frac{1}{2} \cdot 1_{[0,\frac{1}{n})}$. So, $d_2(\mu) = 4 \cdot \int_T h_n^{(0)}h_n^{(\pi)} \, d\widetilde{\mu}_n = \frac{4}{2\pi} \cdot \int_0^\frac{1}{n} \frac{1}{4} \cdot 2n \, d\gamma = \frac{1}{\pi}$ and $d_2(\mu_0)=0$.

(b) Define $\nu :=\lambda$, $\mu_n(d\gamma) := (1+\cos ((2n+1)\gamma) \, \lambda(d\gamma)$, $\gamma \in [0,2\pi)$. We obtain  w-$\lim \mu_n =\lambda =: \mu_0$, $\widetilde{w}_n = 2 \widetilde{\lambda}$, $h_n^{(0)} (\gamma) = \frac{1}{2} (1+\cos ((2n+1)\gamma))$, $h_n^{(\pi)} (\gamma) = \frac{1}{2}(1-\cos ((2n+1)\gamma))$, $\gamma \in [0,\pi)$. Therefore, $d_2(\mu_n) = \frac{4}{2\pi} \int_0^\pi \frac{1}{4} (1+\cos ((2n+1)\gamma))(1-\cos ((2n+1)\gamma)) \, d\gamma = \frac{1}{4}$ and $d_2(\mu_0)=1$.
\end{example}

\begin{theorem}    \label{t7.6}
 Let $\rm m$-$\lim \mu_n =\mu_0$ and $\mu_0 << \mu_n$, $n \in\mathbb N$. 

(i) If ${\rm card}A < \infty$, then (R3) is satisfied for $S=x+H$ and $p=2$.

(ii) If ${\rm card}A=2$, then (R3) is satisfied for $S=x+H$ and $p \in (1,\infty)$.
\end{theorem}

\begin{proof}
(i) Since $|\widetilde{\phi}_k| \leq 1$ $\, \widetilde{\mu}_0$-a.e., from (\ref{(e7.7)}), Lemma \ref{l7.1} and Theorem \ref{t7.2}(i) we derive $\lim \int |\phi_{0,2}-\phi_{n,2} |^2 \, d\mu = \lim \int_T | \widetilde{\phi}_0- \widetilde{\phi}_n |^2\, d\widetilde{\mu} =0$ by Lebesgue's dominated convergence theorem.

(ii) For $\mu_k$ define a function $\widetilde{\psi}_k$ analogously to the function $\widetilde{\psi}$ in (\ref{(e7.4)}). Then $| \widetilde{\psi}_k | \leq 1$ and $\widetilde{\mu}_0$-$\lim \widetilde{\psi}_k = \widetilde{\psi}_0$. So, the proof can be completed similarly to the proof of (i).
\end{proof}

Our next example unveils that the condition ${\rm card}A < \infty$ cannot be omitted in theorem \ref{t7.6}(i).

\begin{example}  \label{e7.7}
Let $G=\mathbb R$, $H=\mathbb Z$ and $x=\frac{1}{2}$. The dual group $\Gamma$ of $\mathbb R$ can be identified with $\mathbb R$ setting $e_y(\gamma) := {\rm e}^{iy\gamma}$, $y,\gamma \in \mathbb R$. The annihilator $A$ of $\mathbb Z$ is the group $2\pi\mathbb Z$, and we select $T=[0,2\pi)$. Set $\nu=\delta_0 + \sum_{j=1}^\infty j^{-2}\delta_{(2j+1)2\pi}$, $\mu_n=\delta_0+\delta_{(2n+1)2\pi}$, $n \in \mathbb N$. Then m-$\lim \mu_n = \delta_0 = :  \mu_0 = \widetilde{\mu}_0$, $\widetilde{\mu}_n = 2\delta_0$, and it is enough to study the behaviour of $\widetilde{\phi}_k(0)$. We have $\widetilde{\phi}_0(0) = {\rm e}^{-i\pi 0} h^{(0)} = 1$ and $\widetilde{\phi}_n(0) = {\rm e}^{-i\pi 0} h_k^{(0)} + \exp (-\frac{i}{2}(2n+1)2\pi) h_n^{((2n+1)2\pi)} (0) = \frac{1}{2}-\frac{1}{2} =0$, and hence, $\int_\Gamma |\phi_{0,2} - \phi_{n,2} |^2 \, d\mu = \int_T | \widetilde{\phi}_0 - \widetilde{\phi}_n|^2 \, d\widetilde{\mu}_0 = 1$, $n \in \mathbb N$.
\end{example}

\begin{theorem}   \label{t7.8}
Let $\lim \| \mu_0 - \mu_n \| = 0.$
\begin{itemize}
\item[(i)] The relation (R1) is satisfied for $S=x+H$ and $p=2$.
\item[(ii)]  If ${\rm card}A = 2$ the relation (R1) is satisfied for $S=x+H$ and $p \in (1,\infty)$.
\item[(iii)] The relation (R3) is satisfied for $S=x+H$ and $p=2$. 
\item[(iv)]  If ${\rm card}A =2$ the relation (R3) is satisfied for $S=x+H$ and $p \in (1,\infty)$.
\end{itemize}
\end{theorem}

\begin{proof}
Since item (iv) is merely a weaker assertion as Theorem \ref{t7.6}(ii) and since (i)  and (ii) are immediate consequences of Theorem \ref{t7.4}(i) and (ii), resp., we have only to prove assertion (iii). According to Theorem \ref{t7.6}(i) it is enough to consider the case in which $A= \{ \alpha_j : j \in {\mathbb N} \}$ is an infinite set. For $\varepsilon \in (0,\infty)$ there exists a number $m \in \mathbb N$ such that 
\begin{equation}   \label{(e7.12)}
   \int_T \big| \sum_{j > m} e_x(-\alpha_j) h_0^{(\alpha_j)} \widetilde{w}_0 \big| \, d\widetilde{\nu} = 
\int_T \sum_{j >m} h_0^{(\alpha_j)} \widetilde{w}_0 \, d\widetilde{\nu} = \mu_0 \bigg( \cup_{j > m} (\alpha_j+T) \bigg) < \varepsilon \, .
\end{equation}
Since $\lim \mu_n \bigg( \cup_{j > m} (\alpha_j+T) \bigg) = \mu_0 \bigg( \cup_{j > m} (\alpha_j+T) \bigg)$ there exists a number $n_1 \in \mathbb N$ such that
\begin{equation}   \label{(e7.13)}
    \int_T \big| \sum_{j > m} e_x(-\alpha_j) h_n^{(\alpha_j)} \widetilde{w}_n \big| \, d\widetilde{\nu} \leq 
\int_T \sum_{j >m} h_n^{(\alpha_j)} \widetilde{w}_n \, d\widetilde{\nu} = 
\mu_n \bigg( \cup_{j > m} (\alpha_j+T) \bigg) < 2\varepsilon
\end{equation}
for all $n \geq n_1$. From $\lim \|\mu_0-\mu_n\| =0$ we can conclude that $\lim \int_T |h_0^{(\alpha)} \widetilde{w}_0 - h_n^{(\alpha)} \widetilde{w}_n | \, d\widetilde{\nu} =0$ for all $\alpha \in A$. Consequently, there exists a number $n_2 \in \mathbb N$ such that
\begin{equation}   \label{(e7.14)}
   \int_T \big|\sum_{j=1}^m e_x(-\alpha_j) h_0^{(\alpha_j)} \widetilde{w}_0 - \sum_{j=1}^m e_x(-\alpha_j) h_n^{(\alpha_j)} \widetilde{w}_n \big| \, d\widetilde{\nu} \leq 
\int_T \sum_{j=1}^m \big| h_o^{(\alpha_j)} \widetilde{w}_0 -h_n^{(\alpha_j)} \widetilde{w}_n \big| \, d\widetilde{\nu} < \varepsilon
\end{equation}
for all $n \geq n_2$. Taking into account (\ref{(e7.1)}),   (\ref{(e7.2)}),  (\ref{(e7.12)}),  (\ref{(e7.13)}) and  (\ref{(e7.14)}), we obtain 
\begin{equation}   \label{(e7.15)}
   \int_T | \widetilde{\phi}_0 \widetilde{w}_0 - \widetilde{\phi}_n \widetilde{w}_n | \, d\widetilde{\nu} < 4 \varepsilon
\end{equation}
and
\begin{equation}   \label{(e7.16)}
   \int_T | \widetilde{w}_0 -\widetilde{w}_n | \, d\widetilde{\nu} = \int_T \big| \sum_\alpha h_0^{(\alpha)} \widetilde{w}_0 -  \sum_\alpha h_n^{(\alpha)} \widetilde{w}_n \big| \, d\widetilde{\nu} < 4 \varepsilon
\end{equation}
for all $n \geq n_0 := \max(n_1,n_2)$. Now, Lemma \ref{l7.1}, (\ref{(e7.15)}) and (\ref{(e7.16)}) imply that
\begin{eqnarray*}
   \int_\Gamma | \phi_{0,2}-\phi_{n,2} |^2 \, d\mu_0 & = & \int_T |\widetilde{\phi}_0 - \widetilde{\phi}_n |^2 \, d\widetilde{\mu}_0 \\
   & \leq & \int_T | \widetilde{\phi}_0 - \widetilde{\phi}_n | ( |\widetilde{\phi}_0| + |\widetilde{\phi}_n|) \, d\widetilde{\mu}_0 \\
   & \leq & 2 \cdot \int_T | \widetilde{\phi}_0 - \widetilde{\phi}_n | d\widetilde{\mu}_0 \\
   & = & 2 \cdot \int_T | \widetilde{\phi}_0 \widetilde{w}_0 - \widetilde{\phi}_n \widetilde{w}_0 | \, d\widetilde{\nu} \\
   & = & 2 \cdot \big( \int_T | \widetilde{\phi}_0 \widetilde{w}_0 - \widetilde{\phi}_n \widetilde{w}_n | \, d\widetilde{\nu}  +
                    \int_T | \widetilde{\phi}_n \widetilde{w}_n - \widetilde{\phi}_n \widetilde{w}_0 | \, d\widetilde{\nu} \big) \\
   & < & 2 \cdot (4 \varepsilon + \int_T | \widetilde{\phi}_n | |\widetilde{w}_0 - \widetilde{w}_n | \, d\widetilde{\nu}) \\
   & < & 2 \cdot (4 \varepsilon +4 \varepsilon)\\
   & = &  16 \varepsilon
\end{eqnarray*}
for all $n \geq n_0$.
\end{proof}

\begin{corollary}   \label{c7.9}
Let $S_k=\{ y_1, ... ,y_k\}$ be a finite subset of $G$ such that any trigonometric polynomial of the form $\sum_{j=1}^k a_j e_{y_j}$, $a_j \in \mathbb C$, $\sum_{j=1}^k |a_j|^2 > 0$, does not belong to $[{\mathcal T}(x+H)]_{\mu_0,2}$. If $\lim \| \mu_0-\mu_n \| =0$, then the relations (R1) and (R3) are satisfied for $S=(x+H) \cup S_k$ and $p=2$. 
\end{corollary}

\begin{proof}
Proposition \ref{p3.14}, Theorem \ref{t7.2}(i) and Theorem \ref{t7.8}(iii) yield (R3). Relation (R1) then follows from Corollary \ref{c3.12}.
\end{proof}

%============================================%

\end{document}